\documentclass{amsart}

\usepackage{amssymb}
\allowdisplaybreaks

\numberwithin{equation}{section}

\newtheorem{theorem}{Theorem}[section]
\newtheorem{proposition}{Proposition}[section]
\newtheorem{lemma}[theorem]{Lemma}
\newtheorem{corollary}[theorem]{Corollary}
\newtheorem{claim}[theorem]{Claim}

\theoremstyle{definition}
\newtheorem{definition}[theorem]{Definition}

\theoremstyle{remark}
\newtheorem{remark}[equation]{Remark}

\usepackage{hyperref} 
\hypersetup{
    colorlinks=true,       
    linkcolor=blue,          
    citecolor=magenta,        
    filecolor=magenta,      
    urlcolor=cyan           
}  

\usepackage[msc-links,shortalphabetic]{amsrefs}

\def\QSet{\mbox\textup{kern.24em \vrule width.03em height1.48ex depth-.051ex \kern-.26em Q}}
\def\S{{\mathbf S}}
\def\M{{\mathcal M}}
\def\PSet{\mbox{\rm I\kern-.22em P}}
\def\B{{\mathbf B}}
\def\P{{\mathbf P}}

\def\Q{{\mathbf Q}}
\def\I{{\mathbf I}}
\def\supp{{\rm supp}}

\def\R{{\mathbb R}}
\def\T{{\mathbf T}}

\def\I{{\mathbf I}}
\def\J{{\mathbf J}}

\def\<{{\langle}}
\def\>{{\rangle}}
\def\dist{\operatorname{dist}}
\def\size{\operatorname{size}}

\def\density{\operatorname{density}}

\def\Z{{\mathbb Z}}
\def\supp{{\,supp}}
\def\f{{{\bf f}}}
\title {Weighted bounds for variational Fourier series}

\author{Yen Do}
\address{ Department of Mathematics, Yale University, New Haven CT 06511, USA}
\email {yenquang.do@yale.edu}
\thanks{Research supported in part by grant NSF-DMS 1201456. Studia Math, to appear.}

\author{Michael Lacey}  
\address{ School of Mathematics, Georgia Institute of Technology, Atlanta GA 30332, USA}
\email {lacey@math.gatech.edu}
\thanks{Research supported in part by grant NSF-DMS 0968499 
	and  a grant from the Simons Foundation (\#229596 to Michael Lacey).}

\begin{document}
\begin{abstract}
{For $1<p<\infty$ and for weight $w$ in $A_p$, we show that the $r$-variation of the Fourier sums of  any function $f$ in $L^p(w)$ is finite a.e. for $r$ larger than a finite constant depending on $w$ and $p$. The fact that the variation exponent depends on $w$ is necessary. This strengthens previous work of Hunt--Young and is a weighted extension of a variational Carleson theorem of Oberlin--Seeger--Tao--Thiele--Wright. The proof uses weighted adaptation of phase plane analysis and a weighted extension of a variational inequality of L\'epingle.}
\end{abstract}

\maketitle

\section{Introduction} \label{s.intro}
For a measurable function $f$ on $[0,1]$, let $Sf$ denote the maximal Fourier sum:
 $$Sf(x) := \sup_{n} |(S_n f)(x)|  \qquad  , \qquad S_n f(x):=\sum_{|k|< n} \widehat f(k) e^{i2\pi kx}  \,.$$
Here, $ \widehat f (k) = \int _{0} ^{1} f (x) \operatorname e ^{-i 2 \pi kx} \; dx$ is the $ k$th Fourier coefficient, and 
by   convention,  $S_n f=0$ for $n\le 0$. (Here we use strict inequality $|k|<n$ in the definition of $S_n$ for the convenience of the transference argument in Section~\ref{s.transference}.) 

By the Carleson--Hunt theorem \cites{MR0199631,MR0238019}, $S$ is bounded on $L^p$ for $1<p<\infty$, which leads to a.e.\thinspace convergence of the Fourier series of functions in $L^p$. See also Sj{\"o}lin \cite{MR0241885} for the Walsh case, and \cites{fefferman, lacey-thiele-carleson} for alternative proofs. More quantitative information about the convergence rate of Fourier series has been obtained by Oberlin--Seeger--Tao--Thiele--Wright  \cite{oberlin-et-al}, via bounds on a strengthening of $S$. To formulate this strengthening of $S$, we first recall the $r$-variation norm of a sequence $(a_n)_{n\in \Z}$. If $0< r <\infty$ then
$$\|(a_n)\|_{V^r} := \sup_{M, N_0<\dots<N_M} \Bigl[|a_{N_0}|^r + \sum_{j=1}^M |a_{N_j}  -  a_{N_{j-1}}|^r\Bigr]^{1/r}\,,$$
and for $r=\infty$ we have $\|(a_n)\|_{V^\infty} = \sup_n |a_n|$. It is clear that if $\|(a_n)\|_{V^r}$ is finite for some $r<\infty$ then  $(a_n)$ is a Cauchy sequence and therefore is convergent; the finiteness of $\|a\|_{V^r}$ may be considered as a quantitative measurement of the convergence rate of  $(a_n)$. The variational strengthening of $S$ considered in \cite{oberlin-et-al} is  the following operator
\begin{equation}\label{e.varsum}
S_{[r]} f(x) = \sup_{M, N_0<\dots<N_M} \Bigl[\sum_{j=1}^M |S_{N_j} f(x) -  S_{N_{j-1}}f(x)|^r\Bigr]^{1/r}\,,
\end{equation}
and it was shown in \cite{oberlin-et-al} that, for $1<p<\infty$, $S_{[r]}$ is bounded in $L^p([0,1])$ if $r>\max(2,p')$.

Convergence of Fourier series in non-Lebesgue settings was also considered by Hunt--Young~\cite{hunt-young}, where it was shown that $S$ is bounded on $L^p(w)$ for any $A_p$ weight $w$, $1<p<\infty$. See also \cite{MR2115460} for extensions to more generalized settings. Recall that a positive a.e.\thinspace  weight $w$ is in $A_p$ if uniformly over intervals $I$ we have
$$ [w] _{A_p} := \sup_{I} \frac{1}{|I|}\Big[\int_I  w(x)dx\Big] \Big[\frac{1}{|I|}\int_I  w(x)^{-1/(p-1)} dx \Big]^{p-1} < \infty  \, .$$

Our aim in this paper is to strengthen the results of \cite{hunt-young} and \cite{oberlin-et-al} by considering weighted estimates for $S_{[r]}$.

\begin{theorem}\label{t.maincor} Let $1<p<\infty$ and $w\in A_p$. Then there is an $ R= R (p,[w] _{A_p})<\infty$ such that for all $r\in (R,\infty]$ we have
\begin{equation}\label{e.maincor} \|S_{[r]} f \|_{L^p([0,1],w)} \le C \|f\|_{L^p([0,1],w)}  
\end{equation}
for some constant $C$ depending only on $w$, $p$, $r$. 
\end{theorem}
As remarked above, Theorem~\ref{t.maincor} gives more quantitative information about the  convergence of Fourier series than \cite{hunt-young} (which corresponds to the endpoint $r=\infty$).  Theorem~\ref{t.maincor} follows  from

\begin{theorem}\label{t.main} Let $1 < p < \infty $ and $w\in A_q$ for some $q\in [1,p)$. Then for $r>\max(2q,\frac{pq}{p-q})$ it holds that 
\begin{equation}\label{e.main}\|S_{[r]} f\|_{L^p([0,1],w)} \le C \|f\|_{L^p([0,1],w)}  
\end{equation}
for some constant $C$ depending only on $w$, $p$, $q$, $r$. 
\end{theorem}

We derive Theorem~\ref{t.maincor} from Theorem~\ref{t.main}. Let $1<p<\infty$ and $w\in A_p$. Since the $A_p$ condition is an open condition, we have $w\in A_q$ for some $1<q<p$ (see e.g. \cite{MR0293384}). Then \eqref{e.maincor} follows from applying Theorem~\ref{t.main}.  

We would like to point out that,  in
the conclusion of Theorem~\ref{t.maincor}, the  variation exponent must depend upon $w \in A_p$.  Indeed, suppose towards a contradiction that there is some $p\in (1,\infty)$ such that \eqref{e.maincor} holds for every $w\in A_p$ and for fixed $r \in (0, \infty)$. Using the fact that variation-norm decreases as $r$ increases, we may assume that $r>1$.  Then,  $S_{[r]}$ is sublinear, and an application of the Rubio de Francia extrapolation theorem shows that the same 
inequality (with the same $r$) would have to hold for $w$ being the Lebesgue measure and \emph{all} $p \in (1,\infty)$, contradicting an 
example in \cite{oberlin-et-al}*{Section 2}. We also remark that in the Lebesgue setting when $w\equiv 1 \in A_1$ the range of $r$ in Theorem~\ref{t.main} is sharp.

Our proof of Theorem~\ref{t.main} extends our previous work in \cite{weightWalsh} on a Walsh--Fourier model of $S_{[r]}$ and at the same time is a weighted extension of \cite{oberlin-et-al}. The proof uses two new ingredients:  weighted analysis on the Fourier phase plane, and a weighted extension of a classical variational inequality of L\'epingle (Lemma~\ref{l.weightLepingle}). The weighted adaptation of analysis on the Fourier phase plane in our proof follows closely the adaptation in \cites{weightWalsh}, modulo (substantial) technicalities arising from the lack of perfect localization of Fourier wave packets. In particular, our approach is different from the elegant argument in \cite{hunt-young} where  a good-$\lambda$ argument was used to deduce weighted bounds for $S$ from the Carleson--Hunt theorem. It is not hard to see that a naive adaptation of the good-$\lambda $ approach in \cite{hunt-young} does not apply to the variation-norm Carleson operator. Our approach is inspired by an argument of Rubio de Francia \cite{rubio-de-francia}, though it is easier to see this inspiration in the dyadic setting of \cite{weightWalsh}.
 We anticipate that the weighted phase plane analysis in our proof will be useful in a variety of open problems involving weighted bounds for multilinear operators with oscillatory nature, where a naive adaptation of the approach in \cite{hunt-young} seems not applicable\footnote{We would like to point out that Xiaochun Li \cite{XiaochunLi} has some unpublished results about weighted estimates for the bilinear Hilbert transform.}.  It is interesting to compare our paper with that of Bennett--Harrison \cite{MR2880218}.

\subsection{Notational convention}\label{s.notation}

(i) Henceforth, we work on the real line $ \mathbb R $, and 
set $ \widehat f (\xi )= \int f (x) \operatorname e ^{-i 2 \pi x \xi  } \; dx  $.

\noindent (ii)
For any $1\le t<\infty$ we will denote by $\M_t f$ the $L^t$ Hardy-Littlewood maximal function, and by $\M_{t,w}f$ the weighted $L^t$ maximal function
$$\M_{t,w}f (x) = \sup_{I: x\in I} \Big(\frac{1}{w(I)} \int_I |f(x)|^t w(x) dx \Big)^{\frac 1 t} \,.$$

\noindent (iii) The dyadic intervals $ \mathbf D$ will play a distinguished role.  
We denote by $f^{\sharp}$ the dyadic sharp maximal function of $f$, namely 
\begin{equation*}
f ^{\sharp} (x) := \sup _{I \in \mathbf D} 1_{I} (x) \lvert  I\rvert ^{-1} \int _{I} \Bigl\lvert  f -  \lvert  I\rvert ^{-1} \int _{I} f (y) \; dy  \Bigr\rvert \; dx \,.
\end{equation*}
All BMO norms, unless otherwise specified, are dyadic BMO norms, namely $ \lVert f\rVert_{BMO} = \lVert f ^{\sharp}\rVert_{\infty }$. 
An important inequality for this paper is the familiar estimate 
\begin{equation}\label{e.sharp<}
\lVert  \phi \rVert_{L ^{p} (w)} \simeq \lVert \phi ^{\sharp}\rVert_{ L ^{p} (w)}\ \ , \qquad w\in A_p \,. 
\end{equation}

\noindent (iv) For any interval $I$ and $c>0$ we denote by $cI$ the interval with length $c|I|$ and with the same center as $I$.
This should not be confused with $c(I)$ which will denote the center of $I$.  
A standard property of an  $w \in A_p $ weight is that it is  doubling. 
There exists $\gamma=\gamma(w)$ such that for any interval $I$ and any $k\ge 0$ it holds that
\begin{equation}\label{e.doubling}
w(2^k I) \le 2^{\gamma k} w(I) \,.
\end{equation}

\noindent (v) For any set $G$ we denote $w(G)=\int_G w(x)dx$.

\subsection{Transference to a singular integral form}\label{s.transference}
Using a weighted variant of a transference argument in \cite{oberlin-et-al}*{Appendix A},  it is not hard to see that Theorem~\ref{t.main} follows from Theorem~\ref{t.mainsingular} stated below. In Theorem~\ref{t.mainsingular}, we define
\begin{equation}\label{e.varsingint}
C_{[r]} f(x) := \sup_{K, N_0<\dots<N_K} \Big(\sum_{j=1}^K |\int_{N_{j-1}}^{N_j} \widehat f(\xi)e^{i2\pi x\xi} d\xi|^r\Big)^{1/r} \,. 
\end{equation}
\begin{theorem}\label{t.mainsingular} Let $1 < p < \infty $ and $w\in A_q$ for some $q\in [1,p)$. Then for $r > \max (2q,\frac{pq}{p-q})$ it holds that 
\begin{equation}\label{e.mainsingular}\|C_{[r]} f\|_{L^p(\R, w)} \le C \|f\|_{L^p(\R, w)}  
\end{equation}
for some constant $C$ depending only on $w$, $p$, $q$, $r$. 
\end{theorem}
For the reader's convenience, we include details of the transference argument.  

For any $K\ge 1$ and $m\ge 1$, let $I_{m,K}$ be the set of all non-decreasing sequences of length $K+1$ in $\{0,\dots, m\}$. For each such sequence $\vec{N}=(N_0\le \dots \le N_K)$ we construct the variation sum 
\begin{equation}\label{e.variationS}
S_{\vec{N}}f = (\sum_{j=1}^K |S_{N_j}f - S_{N_{j-1}}f|^r)^{1/r} \,.
\end{equation}
Since the set $I_{m,K}$ is bigger when  $m$ or $K$ is larger, by two applications of the monotone convergence theorem it suffices to show that
$$\|\sup_{\vec N \in I_{m,K}} S_{\vec N} f\|_{L^p ([0,1],w)} \le C\|f\|_{L^p([0,1],w)}\,,$$
where the implicit constant is  uniform over $m$ and $K$. Let $\sigma = w^{1-p'}$. Then the above inequality has the following equivalent dual form: for $f$ defined on $[0,1]$ and for $g$ defined on $[0,1]\times I_m \times \{1,\dots,K\}$ (we will write $g_{\vec N,j}(x)$ to denote $g(x,\vec N, j)$),
$$\int_0^1 f(x) \sum_{\vec N \in I_{m,K}} \sum_{j=1}^K  \Big[(S_{N_j} - S_{N_{j-1}}) g_{\vec N, j}\Big](x)   dx
$$
\begin{equation}\label{e.maindual}
\le C \|f\|_{L^p([0,1],w)}  \Big \|\sum_{\vec N \in I_{m,K}} (\sum_{j=1}^K |g_{\vec N, j}|^{r'})^{1/r'}\Big \|_{L^{p'}([0,1],\sigma)} \,.
\end{equation}
 To prove \eqref{e.maindual}, we may assume without loss of generality that $f$ and $g_{\vec N, j}$  are trigonometric polynomials for any $\vec N \in I_m$ and $1\le j \le K$. 

For any $N\ge 0$ let $C_N$ be the Fourier multiplier operator on $L^2(\R)$ whose symbol is the characteristics function of $\{-(N-1/3) \le \xi \le N-1/3\}$ (by definition $C_N \equiv 0$ if $N<1/3$). Let $\delta(x) = e^{-\pi x^2}$ and $\delta_M(x) = \delta(x/M)$.  

By standard transference theory (see e.g. \cite{stein-weiss}*{page 261}), for any integer $N$ and any $1$-periodic trigonometric polynomials $P$, $Q$ we have
$$\int_0^1 P(x)  S_N Q(x)   dx = \lim_{M\to \infty} \frac 1 M \int_{\R} P(x) \delta_{M/\alpha} C_N(\delta_{M/\beta} Q) (x) dx \ \  ,$$
for any $\alpha,\beta \in (0,1)$ such that $\alpha^2 + \beta^2 = 1$. We  take $\alpha = \beta = 1/\sqrt 2$. It follows that the left hand side  of \eqref{e.maindual} is the same as
$$=\lim_{M\to \infty} \frac 1 M \int_{\R} f(x) \delta_{M/{\alpha}}(x) \sum_{\vec N\in I_m} \sum_{j=1}^K \Big[(C_{N_j} - C_{N_{j-1}})(\delta_{M/\beta} g_{\vec N, j})\Big](x) dx\,.$$
It follows from Theorem~\ref{t.mainsingular} that the analogue of \eqref{e.maindual} for $C_N$'s holds, thus the above limit is bounded above by
\begin{equation}\label{e.limR}
\le C \limsup_{M\to\infty} \frac 1 M \| f\delta_{M/\alpha}\|_{L^{p}(\R, w)}  \Big\|\delta_{M/\beta}\sum_{\vec N \in I_{m,K}} (\sum_{j=1}^K |g_{\vec N, j}|^{r'})^{1/r'} \Big \|_{L^{p'}(\R,\sigma)} \,.
\end{equation}
Since $w\in A_q\subset A_p$, we have $\sigma = w^{1-p'} \in A_{p'}$ and in particular both $w$ and $\sigma$ are doubling weights. On the other hand, it follows from exponential decay of $\delta$ that for any doubling measure $\mu$ and any $1<q<\infty$ and any $1$-periodic function $h$
$$\sup_{M\ge 1} \frac 1{M^{1/q}}\|\delta_M h \|_{L^q(\R,\mu)} \le C \|h\|_{L^q([0,1],\mu)} \,.$$
Using this observation, \eqref{e.maindual} follows immediately from \eqref{e.limR}.

We take up the proof of  Theorem~\ref{t.mainsingular} below. 

\section{Discretization}\label{s.discretization}
In this section we reduce the task of proving \eqref{e.mainsingular} to proving similar bounds on model operators.  Consider absolute constants $C_2 \in [1,\infty)$ and $C_3\in (0,C_2)$ and $C_{2,1}, C_{2,2}, C_1$ in $[C_2,\infty)$. Constants with these properties are called admissible.

\subsection{Tiles and bitiles}
In this paper, a tile is a dyadic rectangle of area $1$, which we will write $p=I_p\times \omega_p$ and refer to $I_p$ as the spatial interval and $\omega_p$ as the frequency interval of $p$. By a bitile $P$ we mean a rectangle $I_P\times \omega_P$ that contains (as subsets) two tiles $P_1$ and $P_2$ such that they share the same (dyadic) spatial interval $I_P$ and
$$\text{supp} C_2 \omega_{P_1} \le \inf C_2 \omega_{P_2} \,, \ \  |\omega_P| \le C_1 (|\omega_{P_1}|+|\omega_{P_2}|) \,,$$
$$\omega_P = \text{convex hull}(C_{2,1}\omega_{P_1}\cup C_{2,2}\omega_{P_2}) \,.$$
The classical setting (see e.g. \cite{lacey-thiele-carleson}) when a bitile is a dyadic rectangle of area $2$ is the special case of our general setting when $C_2=C_{2,1}=C_{2,2}=C_1=1$.

We say that two bitiles $P$ and $P'$ are disjoint if they are disjoint in the phase plane. Denote by $\widetilde \omega_P$ the convex hull of $C_2 \omega_{P_1} \cup C_2 \omega_{P_2}$, clearly $\widetilde \omega_P \subset \omega_P$. In this paper, whenever we talk about a bitile collection it shall be assumed that the implicit constants above are the same for any two bitiles.

\subsection{Fourier wave packets}
For every tile $p=I_p\times \omega_p$, a function $\phi_p$ is called a Fourier packet adapted to $p$ if  $\supp (\widehat \phi_p)\subset C_3 \omega_p$, furthermore for any $N>0$ and $n\ge 0$ it holds (for some $C_{N,n}$ depending only on $N$ and $n$) that
\begin{equation}\label{e.fourierpacket}
|\frac {d^n}{dx^n}\phi_p(x)| \le C_{N,n} \frac 1{|I_p|^{1/2+n}} (1+\frac{|x-c(I_p)|}{|I_p|})^{-(N+n)}
\end{equation}
here recall that $c(I_p)$ denotes the center of $I_p$. In a family of Fourier packets, we will assume that the involved implicit constants are uniform.

\subsection{Discretization and the model operators}

For any $r\in [1,\infty)$ and any finite collection $\P$ of bitiles, let
$$C_{r,\P} f := \sup_{K, N_0<\dots<N_K} \Big(\sum_{j=1}^K |\sum_{P \in \P} \<f,\phi_{P_1}\>  \phi_{P_1} 1_{\{N_{j-1}\not\in \omega_P, N_j\in \omega_{P_2}\}}|^r\Big)^{1/r} \,.$$
A symmetric variant of $C_{r,\P}$ can be obtained by changing the limiting condition involving $N_j$, $N_{j-1}$ in the above definition to $\{N_{j-1}\in \omega_{P_1}, N_j \not\in \omega_{P}\}$.

Without loss of generality, we assume in the rest of the paper  that $2q<r<\infty$ and $q \in (1,\infty)$. Via a discretization argument in \cite{oberlin-et-al}, which we summarize below, Theorem~\ref{t.mainsingular} follows from  the Theorem below and its symmetric variant (whose proof is completely analogous).

\begin{theorem}\label{t.discretizedmain} There is a constant  $C<\infty$ independent of $f$ and $\P$ such that
\begin{equation}\label{e.discretizedmain}
\|C_{r,\P}f\|_{L^p(w)} \le C \|f\|_{L^p(w)} \,
\end{equation}
 for any finite collection $\P$ of bitiles and any $p \in (q,\infty)$  such that $1/p<1/q-1/r$.
\end{theorem}
 
\begin{proof}[Discretization] We sketch the main ideas of our weighted adaptation of the discretization argument in \cite{oberlin-et-al}*{Section 3}. For each interval $(a,b)$ with non-dyadic endpoints, let $\J$ be the collection of maximal dyadic intervals in $(a,b)$ such that $\dist(J,a), \dist(J,b)\ge |J|$. It is not hard to see that $\J$ partitions $(a,b)$, and  the ratio between two adjacent elements of $J$ are at most $2$. By direct examination, it follows that there are $O(1)$ possible mutually exclusive  
scenarios involving relative locations of $J$ inside $(a,b)$, and these scenarios are characterized by the following information:
\begin{itemize}
\item  whether $J$ is the left or right child or its dyadic parent, 
\item the distance from $a$ to $J$, which could be arbitrarily large,
\item the distance from $b$ to $J$, which could be arbitrarily large.
\end{itemize}
More specifically, we may divide $\J$ into $O(1)$ disjoint subsets of the following type: If $m,n,k$ are bounded positive integers and $\textup{\emph{side}} $ is $\textup{\emph{left}}$ or $\textup{\emph{right}}$ then we denote by
$\J_{k,m,n,\textup{\emph{side}}}$ the set of all dyadic intervals $J$ such that $J$ is the $\textup{\emph{side}}$-child of its dyadic parent, and $a\in J_{\textup{\emph{low}}}(k,m)$ and $b\in J_{\textup{\emph{high}}}(k,n)$.  
\begin{itemize}  
\item If $k=1$ then  $J_{\textup{\emph{low}}} = J - (m+1)|J|$ and  $J_{\textup{\emph{high}}}=  |J| + (n+1)|J|$.
\item If $k=2$ then  $J_{\textup{\emph{low}}} = J-(m+1)|J|$ and  $J_{\textup{\emph{high}}} = [sup J + n|J|, \infty)$.
\item If $k=3$ then $J_{\textup{\emph{low}}} = (-\infty, \inf J - m|J|]$ and $J_{\textup{\emph{high}}} = |J|+ (n+1)|J|$.
\end{itemize}

The following example of such a partition was given in \cite{oberlin-et-al}, we include this example for the convenience of the reader. Below are the values of $(k,m,n,\textup{side})$:
\begin{align*}
&\{(1,2,1,\textup{left}), (1,2,2,\textup{left}), (1,3,1,\textup{left}),  (1,3,2,\textup{left}), (2,1,1,\textup{left}),\\                  &(2,1,1,\textup{right}), (2,2,1,\textup{right}), (3,4,1,\textup{left}), (3,3,1,\textup{right}), (3,4,2,\textup{left})\} \,.
\end{align*}

Since the relative ratio between adjacent intervals in $\J$ are bounded by 2, we may construct nonnegative $L^\infty$ normalized bump functions $\varphi_J$ such that $1_{(a,b)}(\xi) = \sum_{J\in \J} \varphi_J(\xi) $, furthermore $\varphi_J$ is supported inside a $(1+c)$ dilation of $J$  for each $J\in \J$ , here the absolute constant $c>0$ can be taken arbitrarily small. By using a standard Fourier sampling theorem for the Schwartz band-limited function $\mathcal F^{-1}(\widehat f(\xi) \sqrt{\varphi_J})$ (cf. \cite{thiele-CBMS}) we can easily decompose 
$$\widehat f(\xi)\varphi_J(\xi) = \sum_{|I|=1/(2^L|J|)} \<f,\phi_{I\times J}\> \widehat \phi_{I\times J}(\xi)$$
for some positive integer $L=O(1)$ where  $\widehat{\phi_{I\times J}}(\xi) := |I|^{1/2} \sqrt{\varphi_J(\xi)} e^{-2\pi i c(I)\xi}$.
Note that the frequency support of $\phi_{I\times J}$ is inside a $(1+c)$ dilation of $J$ with $c>0$ can be chosen small. Furthermore, it is clear that the collections of functions $(\phi_{I\times J}: |I|=2^{-L}|J|^{-1})$ can be decomposed\footnote{This decomposition ensures that there is only one wave packet associated with each dyadic rectangle of area $1$.} into $O(1)$ families of Fourier wave packets adapted to the tiles in the phase plane.

Let $\P_{\textup{\emph{side}}}$ denote the collection of all dyadic rectangles of area $2^{-L}$ whose frequency interval is the $side$-child of its parent. Then
$$\int_a^b e^{i2\pi x\xi} \widehat f(\xi)d\xi = \sum_{k=1}^3 \sum_{m,n,\textup{\emph{side}}} 
\sum_{p\in \P_{\textup{\emph{side}}}}  \<f,\phi_p\>\phi_p(x) 1_{\{a\in L_p(k,m), b\in U_p(k,n)\}} \,,$$
here the intervals $L_p(k,m)$ and $U_p(k,n)$ are the $J_{\textup{\emph{low}}}$ and $J_{\textup{\emph{high}}}$ of $J=\omega_p$.

Now, under the assumption that $f$ is Schwartz, it is no loss of generality to assume that the sequences $(N_0<\dots < N_K)$ (used in the definition of $C_{[r]}$) does not contain endpoints of dyadic intervals . Performing the above partition on every $(N_{j-1}, N_j)$, it then follows from the triangle inequality that
$$C_{[r]}f \le \sum_{m,n,\textup{\emph{side}}} C_{1,m,n,\textup{\emph{side}}} f(x) + C_{2,m,n,\textup{\emph{side}}} f(x) + C_{3,m,n,\textup{\emph{side}}} f(x)\,,$$
$$C_{k,m,n,\textup{\emph{side}}} f(x) := \sup_{K,(N_j)}(\sum_{j=1}^K |\sum_{p\in \P_{\textup{\emph{side}}}} \<f,\phi_p\>\phi_p(x) 1_{\{N_{j-1}\in  L_p(k,m), N_j\in  U_p(k,m)\} }|^r)^{1/r} \,.$$
It is not hard to see that for each $1\le m,n = O(1)$, we can bound $C_{3,m,n} f(x)$ by a sum of $O_L(1)$ operators of the same nature as $C_{r,\P}$, with appropriate choice of admissible constants $C_1$, $C_2$, $C_{2,1}$, $C_{2,2}$ and $C_3$. Similarly, $C_{2,m,n} f(x)$ can be bounded by a symmetric variant of $C_{r,\P}$. Since any interval $[a,b)$ can be written as $(-\infty, b) \setminus (-\infty, a)$, it is not hard to see that $C_{1,m,n} f(x)$ can be controlled by two operators of the same nature as $C_{3,m,n} f(x)$. Thus, Theorem~\ref{t.mainsingular} follows from Theorem~\ref{t.discretizedmain}. This completes the discretization step.
\end{proof}

Below we set up a linearized variant of $C_{r,\P}$. By duality in $ \ell ^{r}$, to show \eqref{e.discretizedmain}  it suffices consider the following operator (we omit the dependence on $r$ for simplicity):  
$$(C_{\P}f)(x)  = \sum_{j=1}^{K(x)}\sum_{P \in \P} \langle f,\phi_{P_1}\rangle  \phi_{P_1}(x) 1_{\{N_{j-1}(x) \not\in \omega_{P}, \ N_j(x) \in \omega_{P_2}\}}d_j(x)\,,$$
here $K:\R\to \Z_+$, $N_0(x)<\dots < N_K(x)$ and $ \{d_j\}$ are measurable functions, with 
$$ \label{e.d}|d_1(x)|^{r'} +\dots + |d_{K(x)}(x)|^{r'}=  1 \,.$$
For each bitile $P$,  let $d_P(x)$ be $0$ unless there exists  a (clearly unique) $j$ such that $N_{j-1}(x) \not\in \omega_{P}$ and $N_j(x) \in  \omega_{P_2}$, in which case we set $d_P(x)=d_j(x)$. 
For a function $ g$, we note that $\<C_{\P}f, gw\>=B_{\P}(f,g) $, where 
$$B_{\P}(f,g) :=  \sum_{P\in \P} \langle f,\phi_{P_1}\rangle  \langle \phi_{P_1} \, d_P \, , \, gw\rangle  \,.$$ 

We say that $G' \subset G$ is a major subset if $w(G') > w(G)/2$ and we say $G'$ has full measure if $w(G')=w(G)$. Via a standard restricted weak-type interpolation argument \cite{MR1887641}*{Section 2}, Theorem~\ref{t.discretizedmain} follows from the following proposition:
\begin{proposition}\label{p.restrictedweaktype} Let $F$, $G$ be such that $w(F)$, $w(G)<\infty$. Then there are major subsets of $F$ and $G$, denoted respectively by $\widetilde{F}$ and $\widetilde{G}$,  such that:\\
(i) at least one subset has full measure, and\\
(ii) for any $ \lvert  f\rvert  \le 1_{\widetilde F}$ and  $|g| \le 1_{\widetilde G}$ and any finite collection of bitiles $\P$ we have
\begin{equation}\label{e.restrictedweaktype}
B_{\P}(f, g) \le C w(F)^{1/p} w(G)^{1-1/p} \, \, 
\end{equation}
for all $p \in (q,\infty)$  such that $1/p<1/q-1/r$.
\end{proposition}
In the rest of the paper, we will prove Proposition~\ref{p.restrictedweaktype}.

\section{Decomposition of bitile collections}\label{s.decomp}
Without loss of generality we may assume the following separation conditions: 
\begin{itemize}
\item[(S1)] The ratio $\textup{dist}(\omega_{P_1}, \omega_{P_2})/|\omega_{P_1}|$ is constant over $P\in \P$.
\item[(S2)] For any two bitiles $P$ and $P'$,  if $\omega_P \cap \omega_{P'} \ne \emptyset$ and $|I_P| = |I_{P'}|$ then  $\omega_{P}=\omega_{P'}$.
\item[(S3)]\label{i.S3} For any two bitiles $P$ and $P'$, if $|I_P| > |I_{P'}|$ then $|\omega_P| < |\omega_{P'_1}|/K_0$ for some large absolute constant $K_0$ that will be chosen in the proof. (The choice  of $K_0$ is refined a bounded number of times below.)
\end{itemize}
\begin{remark}\label{r.lattice} First, we will require that $K_0 > \frac2{C_2-C_3}$. This means that  for any $1\le i\le 2$, 
if $C_3 \omega_{P_i}\cap C_3\omega_{P'_i} \ne \emptyset$ and $|I_P|>|I_{P'}|$ then $\omega_P \subset C_2\omega_{P'_i}$.
\end{remark}

\subsection{Trees}
In this paper, a finite collection $T$ of bitiles is a \emph{tree} if there exists a dyadic interval $I_T$ and a real number $\xi_T$ such that for any $P\in T$ we have
$$I_P  \subset I_T \qquad \text{and} \qquad \omega_T:=[\xi_T - \frac {1}{2|I_T|}, \xi_T + \frac {1}{2|I_T|})  \subset  \widetilde \omega_P\,.$$ 
$I_T$ will be referred to as the \emph{top interval} of $T$. Similarly, $\xi_T$ and $\omega_T$ will be referred to as the \emph{top frequency} and the \emph{top frequency interval} of $T$.

We say that $T$ is \emph{$2$-overlapping} if $\xi_T \in C_2 \omega_{P_2}$ for every $P\in T$, and we say that $T$ is \emph{$2$-lacunary} if  $\xi_T \not\in C_2\omega_{P_2}$ for every $P\in T$.

It is clear that any tree can be split into two trees, one of each type. Furthermore, the union of two trees with the same $(I_T,\xi_T)$ is a tree and we may use the pair $(I_T,\xi_T)$ for the new tree. If  these two trees are $2$-lacunary then the new tree is also $2$-lacunary.

\begin{remark}\label{r.treelacunary}
By further  requiring that $K_0>\frac {C_3}{2C_1+1}$ in the  separation assumption (S\ref{i.S3}), we obtain the following properties (cf. Remark~\ref{r.lattice}). Let $T$ be a tree and let $P, P'\in T$ be two different bitiles.
\begin{itemize}
\item If $|I_P|=|I_{P'}|$ then $I_P \cap I_{P'} = \emptyset$.
\item If $T$ is $2$-overlapping  and $|I_P|>|I_{P'}|$ then $\omega_P \cap  C_3 \omega_{P'_1} = \emptyset$. 
\item If $T$ is  $2$-lacunary  and $|I_P|>|I_{P'}|$  then $\omega_{P} \cap C_3 \omega_{P'_2} = \emptyset$.
\end{itemize}
\end{remark}

\begin{remark}\label{r.maximalsubtree} If there is a dyadic interval $J$ such that for every $P\in\T$ we have $I_P\subset J$ then we can decompose $T$ into $O(1)$ subtrees, each tree has  $J$ as top interval (the top frequencies of these subtrees are not necessarily the same, but they are $O(1/|J|)$ away from the original $\xi_T$). Essentially, this is because we would have $|\widetilde\omega_{P}|\ge \frac 2{|J|}$ and then one can always partition $T$ into two desired trees depending on the relative position of $\xi_T$ in $\widetilde \omega_P$.
\end{remark}

\subsection{Tile norms}
Below, for any collection $\Q$ of bitiles we denote
$$S_{\Q}f(x)  := \Bigl[\sum_{P\in \Q} \frac{|\langle f,\phi_{ P_1}\rangle|^2}{|I_P|} 1_{I_P}\Bigr] ^{1/2}    \,.$$
\begin{definition}[Size] 
The \emph{size} of a collection $\P$ of bitiles  is  
$$
\size(\P) := \sup _{T \subset \P} 
 w(I_T)^{-\frac 1{2}} 
\|S_T f\|_{L^{2}(w)}  \,.$$
The supremum is over all  $2$-overlapping tree $T \subset \P $.  
\end{definition}
It is clear that for $w\equiv 1$ one recovers the standard definition of size (cf. \cites{lacey-thiele-calderon}).  For any interval $I$, let 
$$\widetilde \chi_I(x) = \Big[1+(\frac{x-c(I)}{|I|})^2\Big]^{-1/2} \,.$$
Note that if $J\subset I$ then $\widetilde \chi_{J} \le C\widetilde \chi_ {I}$, and this estimate will be used 
implicitly in future estimates. 

\begin{definition}[Density] Recall the definition of the  functions  $d_j$ from  \eqref{e.d}. Fix a large constant $D \in (0,\infty)$. The density of a collection $\P$ of bitiles is defined to be
$$\density(\P):= \sup_{T} \Big(\frac{1}{w(I_T)} \int \widetilde \chi_{I_T}^D |g|^{r'}\sum_{j: N_j \in \omega_T} |d_j |^{r'}w  \Big)^{1/r'}\ \ ,$$
here the supremum is over nonempty trees $T\subset \P$.
\end{definition}

Choose $D$ to be very large  depending on $w,p,q,r$ in the proof of Proposition~\ref{p.restrictedweaktype}  in Section~\ref{s.mainargument} (see also the proof of Lemma~\ref{l.density}).  All the implicit constants are allowed to depend on $D$. 

When the elements of $\P$ are disjoint in the phase plane, the following improved notion of density is more useful in future estimates, see also Lemma~\ref{l.improved-tree-est}. 
 
\begin{definition}[Improved Density] The \emph{improved density} of a collection $\P$ of bitiles is defined to be
$$\widetilde{\density}(\P):= \sup_{P\in \P}\Big(\frac{1}{w(I_P)} \int \widetilde \chi_{I_P}^D |g|^{r'}\sum_{j: N_j \in \omega_{P_2}} |d_j |^{r'}w  \Big)^{1/r'}\,.$$
\end{definition}
It is clear that $\widetilde \density(\P)\le C \density(\P)$ for any $\P$.
\subsection{Decomposition by size}

We have the following size bound:
\begin{lemma}\label{l.sizebound} Assume $w\in A_q$. Then for  any $N>0$ there is a constant $C=C(N,q,w) <\infty$ such that for any $\P$
$$\size(\P) \le C \sup_{P\in \P} \Big(\frac{1}{w(I_P)}\int |f|^q \widetilde \chi_{I_P}^N w  \Big)^{1/q}$$
\end{lemma}
The main ingredient in the proof of Lemma~\ref{l.sizebound} is the following John-Nirenberg characterization of size, which is a standard result in the Lebesgue setting (see e.g. \cite{MTTBiestFourier}). The proof of the Lebesgue case of this characterization extends smoothly to the weighted setting (see \cite{weightWalsh}*{Lemma 3.5}), we omit the details.
\begin{lemma}\label{l.BMO} For any $1<p<\infty$ and any collection $\P$ we have
\begin{eqnarray*}
\sup_{T\subset \P} \frac{1}{w(I_T)^{1/p}} \|S_T f\|_{L^p(w)}  &\sim_p& \sup_{T\subset \P} \frac{1}{w(I_T)} \|S_T f\|_{L^{1,\infty}(w)}
\end{eqnarray*}
the suprema are over all $2$-overlapping trees. 
\end{lemma}

\begin{proof}[Proof of  Lemma~\ref{l.sizebound} using Lemma~\ref{l.BMO}]  By decomposing $T$ into smaller subtrees (using Remark~\ref{r.maximalsubtree}), we may assume that $I_T=I_P$ for some $P \in T$. Thus, it suffices to show that 
$$\|S_T f\|_{L^q(w)} \le C\|f\widetilde \chi_{I_T}^N\|_{L^q(w)}\ \  . $$
But $ w\in A_q$, hence $\|S_T f\|_{L^q(w)}  \lesssim \| (S_T f) ^{\sharp}\|_{L^q(w)} $.
Therefore it suffices to show that for any $N<\infty$ we have
\begin{equation}\label{e.M1bound}
(S_T f)^{\sharp} \le C \M_1 (f\widetilde \chi_{I_T}^N) \,.
\end{equation}
For any dyadic interval $J$ let 
$$c_J = (\sum_{P\in T: J\subset I_P} \frac {|\<f,\phi_{P_1}\>|^2}{|I_P|})^{1/2} \,.$$
Then
$$\frac 1{|J|}\int_J |S_T f(x) - c_J|dx \ \  \le  \ \ \Big(\frac 1{|J|}\int_J |S_T f(x)^2 - c_J^2| dx \Big)^{1/2}$$
$$=\frac 1{|J|^{1/2}} \| (\sum_{P\in T:  I_P\subsetneq J} |\<f,\phi_{P_1}\>|^2\frac {1_{I_P}}{|I_P|})^{1/2} \|_2\,. $$
Using the known Lebesgue case of Lemma~\ref{l.sizebound} (see e.g. \cite{MTTBiestFourier}*{Lemma 6.8}), we obtain
$$\frac 1{|J|}\int_J |S_T f(x) - c_J|dx \le C \sup_{P\in T: I_P\subsetneq J} \frac 1{|I_P|} \int |f(x)|\widetilde \chi_{I_P}(x)^{N+4} dx$$
$$\le C \inf_{x\in J\cap I_T}\M_1(f\widetilde \chi_{I_T}^N)(x)\,,$$
and \eqref{e.M1bound} follows immediately.
\end{proof}

We remark that the following bound was proved in the   above proof of Lemma~\ref{l.sizebound}:
\begin{corollary}\label{c.BMO} Assume $w\in A_q$. Then for any $2$-overlapping tree $T$ and any $N>0$ it holds that
$$\|S_Tf\|_{BMO} \le C_N \inf_{x\in I_T} \M_1(f \widetilde \chi_{I_T}^N)(x)$$
here we use the dyadic BMO norm.
\end{corollary}

For convenience, in the rest of the paper we say that a collection $\T$ of $2$-overlapping trees is  well-separated if the following conditions are satisfied:
\begin{itemize}
\item[(i)]  If $T, T'\in \T$ are two different trees, and $P \in T$ and $P' \in T'$ and $|I_P|> |I_{P'}|$ then either $C_3 \omega_{P_1} \cap C_3 \omega_{P'_1} =\emptyset$ or  $I_{P'} \cap I_T = \emptyset$.
\item[(ii)] If $P, P'\in  \bigcup_{T\in \T}T$ are two different bitiles with $|I_P| = |I_{P'}|$ then $I_P\times C_3\omega_{P_1}$ and $I_{P'}\times C_3 \omega_{P'_1}$ are disjoint.
\end{itemize}

\begin{lemma}\label{l.size} Let $\P$ be a collection of bitiles with size bounded above by $2\alpha$, some $\alpha>0$. Then we can find a collection $\T$ of trees  such that:
\begin{itemize}
\item The bitile collection $\P - \bigcup_{T\in \T} T$ has size less than $\alpha$.
\item If another tree collection $\T'$ covers $\bigcup_{T\in \T}T$ then for some $C=C(w)<\infty$
\begin{equation}\label{e.efficient}
\sum_{T\in \T} w(I_T) \le C\sum_{T'\in \T'} w(I_{T'}) \,.
\end{equation}
\item If $q_0\in (q,\infty)$ then there exists $\beta=\beta(p,w,q,q_0)<\infty$ such that  for any $k\ge 0$ and for any $1\le p < \infty$ we have
\begin{equation}\label{e.topinterval}
\| \sum_{T\in \T} 1_{2^k T} \|_{L^p(w)} \le C 2^{\beta k} \alpha^{-2q_0} \|f\|_{L^{2pq_0}(w)}^{2q_0} \,.
\end{equation}
Here $C=C(p,w,q,q_0)<\infty$.
\end{itemize}
\end{lemma}

\begin{proof} For convenience let $a_P=\<f,\phi_{P_1}\>$. We follow the standard algorithm from \cite{lacey-thiele-carleson}. If $\size(\P) \ge \alpha$ then there exists a non-empty $2$-overlapping tree $T_2\subset \P$ such that $\|S_{T_2} f\|_{L^2(w)} \ge \alpha^2 w(I_{T_2})$.  We select such a tree with minimal value of $\xi_{T_2}$\footnote{To be more careful, one can fix a top frequency for each of these trees, and then select one tree  (there are only finitely many of them) whose top frequency is minimal.}, and let $T$ be the maximal tree in $\P$ with top data $(I_{T_2}, \xi_{T_2})$. We then remove from $\P$ the bitiles in $T$ and repeat this argument until the remaining collection of bitiles has size less than $\alpha$. We obtain a collection $\T$ of trees such that
\begin{itemize}
\item $\P - \bigcup_{T\in \T} T$  has size less than $\alpha$;
\item Each $T\in \T$ contains  a $2$-overlapping  subtree $T_2$ such that
\begin{equation}\label{e.T2}
w(I_T) \le C\alpha^{-2} \|S_{T_2}f\|_{L^2(w)}^2 = C\alpha^{-2}\sum_{P\in T_2} |a_P|^2 \frac{w(I_P)}{|I_P|}  \,.
\end{equation}
\end{itemize}	
It then follows from a standard geometrical consideration that  the tree collection $\T_2:=\{T_2: T\in \T\}$ is well-separated when the constant $K$ in (S\ref{i.S3}) is chosen sufficiently large (see also Remark~\ref{r.lattice}). We omit the details.

\underline{Proof of \eqref{e.efficient}:} Assume that $\T'$ covers $\Q:=\bigcup_{T\in \T}T$, without loss of generality we can assume $\bigcup_{T'\in \T'}T' = \Q$. Let $\Q_2 = \bigcup_{T\in \T}T_2$. It follows from \eqref{e.T2} that
\begin{equation}\label{e.sumT}
\sum_{T \in \T} w(I_T) \le C\alpha^{-2} \sum_{P\in \Q_2} |a_P|^2 \frac{w(I_P)}{|I_P|} \,.
\end{equation}
Now, divide each $T'\in \T'$ into three trees, 
$$T'_0 = \{P\in T': \inf C_2 \omega_{P_1} \le \xi_{T'} < \sup  C_3 \omega_{P_1}\}\,,$$
$$T'_1 = \{P\in T': \sup C_3 \omega_{P_1} \le \xi_{T'} < \inf  C_2 \omega_{P_2}\}\,,$$
$$T'_2 = \{P\in T': \xi_{T'} \in  C_2 \omega_{P_2}\} \,.$$
Clearly, $T'_2$ is $2$-overlapping. Since $\size(\P)\le C\alpha$, we have
\begin{equation}\label{e.T'2}
\sum_{P\in T'_2} |\<f,\phi_{P_1}\>|^2 \frac{w(I_P)}{|I_P|} = \|S_{T'_2}f\|_{L^2(w)}^2 \le C\alpha^2 w(I_{T'}) \,.
\end{equation}
On the other hand, since $\T_2$ is well separated, the rectangles $I_P \times [\inf C_2 \omega_{P_1}, \sup C_3 \omega_{P_1})$ with $P\in \Q_2$ are pairwise disjoint in the phase lane. This implies that the  bitiles of  $T'_0\cap \Q_2$ are spatially disjoint (since their frequency intervals overlap). Thus,
\begin{equation}\label{e.T'0}
\sum_{P \in T'_0 \cap \Q_2} |\<f,\phi_{P_1}\>|^2 \frac{w(I_P)}{|I_P|} \le \sum_{P\in T'_0\cap \Q_2} \size(\{P\})^2 w(I_P) \le C\alpha^2w(I_{T'}) \,.
\end{equation}
Next, we show that $T'_1\cap \Q_2$ can be grouped into $O(1)$ collections of $2$-overlapping trees whose top intervals are disjoint. Together with the given assumption on the size of $\P$, this would imply
\begin{equation}\label{e.T'1}
\sum_{P \in T'_1 \cap \Q_2} |\<f,\phi_{P_1}\>|^2 \frac{w(I_P)}{|I_P|} \le C \alpha^2 \omega(I_{T'}) \,.
\end{equation}
Let $M$ be the set of elements of $T'_1\cap \Q_2$ with maximal spatial intervals. The grouping of elements in $T'_1 \cap \Q_2$ can be done as follows:
\begin{itemize}
\item Any element $P\in M$ can be viewed as one $2$-overlapping tree, and we place these single-element trees in to the first tree collection.
\item For any $P \in M$, we show below that we can  place every $P' \in T'_1\cap \Q_2$  such that $I_{P'}\subsetneq I_P$ in $O(1)$ trees sharing the top interval $I_P$.
\end{itemize}
Since the interval $\{I_P, P\in M\}$ are disjoint, it remains to show that if $P' \in T'_1\cap \Q_2$  and $I_{P'}\subsetneq I_P$ then
\begin{equation}\label{e.c2omegap2}
\inf C_2 \omega_{P'_2} < \sup C_2 \omega_{P_2} < \sup C_2 \omega_{P'_2} \,.
\end{equation}
Indeed, since $|\omega_{P'_2}|=\frac 1{|I_{P'}|}\ge \frac 2{|I_P|}$ it follows from \eqref{e.T'1} that we may take $-\frac 1{2|I_P|} + \sup C_2 \omega_{P_2}$ or $\frac 1{2|I_P|} + \sup C_2 \omega_{P_2}$ as the top frequency for these trees.

To see the first inequality in \eqref{e.c2omegap2},  we assume (towards a contradiction) that $\sup C_2 \omega_{P_2} \le \inf C_2 \omega_{P'_2}$. By the selection algorithm, the $2$-overlapping tree $S \in \T_2$ that contains $P$ must be selected before the $2$-overlapping tree $S'$ of $P'$. Now, by definition of 
$T'_1$ we have
$$[\sup C_3\omega_{P_1},\inf C_2\omega_{P_2}) \cap [\sup C_3\omega_{P'_1},\inf C_2\omega_{P'_2}) \ne \emptyset$$ 
(they both contains $\xi_{T'}$). On the other hand, by ensuring the constant $K_0$ is sufficiently large in the separation assumption (S\ref{i.S3}), we  have $\omega_{P} \subset \text{convex hull}(C_2\omega_{P'_1} \cap C_2 \omega_{P'_2})$.  But then $P'$ must be cleared out as part of the maximal tree with the same top data as $S$,  leading to a contradiction. This proves the first half of \eqref{e.c2omegap2}.

To see the second inequality in \eqref{e.c2omegap2}, as before exploit the fact that
$$[\sup C_3\omega_{P_1},\inf C_2\omega_{P_2}) \cap [\sup C_3\omega_{P'_1},\inf C_2\omega_{P'_2}) \ne \emptyset \,.$$ 
By ensuring the constant $K_0$ in the separation assumption (S\ref{i.S3}) is sufficiently large, we have $|\omega_{P'_2}|\ge |\omega_{P_2}|$. As a consequence, if $sup\, C_2 \omega_{P_2} \ge sup \, C_2 \omega_{P'_2}$ then the interval $[sup\, C_3\omega_{P_1},  inf\, C_2\omega_{P_2})$ will be above $inf \, C_2\omega_{P'_2}$, contradicting the nonempty intersection. This completes the proof of \eqref{e.c2omegap2} and hence \eqref{e.T'1}.

Finally, collecting inequalities \eqref{e.T'2} \eqref{e.T'0} \eqref{e.T'1}, we obtain
$$\sum_{P \in T'} |\<f,\phi_{P_1}\>|^2 \frac{w(I_P)}{|I_P|}  \le C\alpha^2 w(I_{T'}) \,.$$
Summing over $T'\in \T'$ and using \eqref{e.sumT}, we obtain the desired estimate \eqref{e.efficient}.  \\

\noindent \underline{Proof of \eqref{e.topinterval}:} 
Fix $k$ and let 
$$N^{[k]} :=\sum_{T\in \T} 1_{2^k I_T} \,.$$
It suffices to show the following good lambda estimate:  given any $L\in (0,\infty)$ there exists $c_0\in (0,\infty)$ and $c \in (0,\infty)$  such that
\begin{equation}\label{e.goodlambda}
w(\{N^{[k]}>\lambda\} \cap E^{[k]}_{\lambda}) \le  \frac 1 {L} w(\{N^{[k]} >\lambda/4\}) \,.
\end{equation}
\begin{equation}\label{e.Elambda}
\text{where} \qquad E^{[k]}_\lambda:=\{\M_{2q,w} f \le c 2^{-c_0 k} \alpha \lambda^{\frac 1{2q_0}} \} \,.
\end{equation}
Indeed, choosing $L$ sufficiently large (depending on $p \in [1,\infty)$) and applying a standard bootstrapping argument, we obtain 
\begin{gather*}
\|\sum_{T\in \T}1_{I_T}\|_{L^p(w)} \le C 2^{O(k)} \alpha^{-2q_0} \|\M_{2q,w}(f)\|_{L^{2pq_0}(w)}^{2q_0}
\\ \le C 2^{O(k)} \alpha^{-2q_0} \|f\|_{L^{2pq_0}(w)}^{2q_0}\,,
\end{gather*}
as desired. Here we have used the fact that $\M_{1,w}$ is bounded from $L^t (w) \to L^t (w)$ for any $1<t<\infty$ and any positive weight $w$; note that we always have $2pq_0>2q$.

To prove \eqref{e.goodlambda}, we use the following estimate which follows from Lemma~\ref{l.Nweak} (see the remark after the Lemma): 
for any dyadic interval $I$ and  $q_0 \in (q,\infty)$ it holds that
\begin{gather}\label{e.NIweak}
w(\{N^{[k]}_I>\lambda/4\}) \le C 2^{O(k)} \alpha^{-2q} \lambda^{-\frac {q} {q_0}} w(I) [\inf_{x\in I} \M_{2q,w} (f\widetilde \chi_I^N)(x)]^{2q}
\\ \notag
\textup{where} \quad N^{[k]}_I := \sum_{T\in \T: I_T\subset I} 1_{I_T} \,.
\end{gather}
Let $\I$ be the collection of  all maximal dyadic intervals of $\{N^{[k]}>\lambda/4\}$. We apply \eqref{e.NIweak} to elements of $\I$ that intersect $E^{[k]}_\lambda$. Let $I$ be one such interval, then it follows from the maximality of $I$ that
$\{N^{[k]}>\lambda\} \cap I$ is a subset of $\{N^{[k]}_I > \lambda/4\}$. Thus,
$$w(\{N^{[k]}>\lambda\} \cap I) \le C  2^{O(k)}\Big[ \alpha^{-2q} \lambda^{-\frac {q}{q_0}} w(I)  \Big] \Big[c 2^{-c_0 k}\alpha  \lambda^{\frac 1{2q_0}} \Big]^{2q}$$
and by choosing $c$ sufficiently small and $c_0$ sufficiently large we obtain
$$w(\{N^{[k]}>\lambda\} \cap I)  \le C  c^{2q} w(I) \le \frac {w(I)}L \,.$$
Summing the above estimates over all $I\in \I$ that intersects $E^{[k]}_\lambda$, we obtain \eqref{e.goodlambda}:
$$w(\{N^{[k]}>\lambda\} \cap E^{[k]}_\lambda) \le \sum_{I\in \I: I\cap E_\lambda^{[k]} \ne \emptyset} w(\{N^{[k]}>\lambda\} \cap I)$$
$$ \le \frac 1 L  \sum_{I \in \I} w(I) = \frac 1 L w(\{N^{[k]}>\lambda/4\}) \,.$$
\end{proof}

\begin{lemma}\label{l.Nweak} Let $I$ be an interval and let $\T$ be a well-separated collection of $2$-overlapping trees such that  for any $T\in \T$ we have $I_T\subset I$, and
\begin{equation}\label{e.largesize}
w(I_T) \le C\alpha^{-2}  \|S_Tf \|_{L^2(w)} \,.
\end{equation}
Then for any $q_0\in (q,\infty)$ and $N>0$ there is  $C=C(q_0,w,N)<\infty$ such that
\begin{gather}
\label{e.Nweak} w(\{N^{[k]}>\lambda\}) \le C2^{O(k)} \Big[\alpha^{-1} \lambda^{-\frac {1} {2q_0}}  \|f\widetilde\chi_I^N\|_{L^{2q}(w)}\Big]^{2q}\,,
\\
\textup{where}\quad N^{[k]}:=\sum_{T\in \T}1_{2^k I_T} \,. 
\end{gather}
The implicit constant in $O(k)$ depends on $w$ and $q$.
\end{lemma}

\begin{remark}\label{r:}
As a consequence of \eqref{e.Nweak}, we obtain
\begin{equation}
\label{e.NweakMaximal} w(\{N^{[k]}>\lambda\}) \le C  2^{O(k)} w(I) \Big[\alpha^{-1} \lambda^{-\frac {1} {2q_0}}  \inf_{x\in I} \M_{2q,w} (f\widetilde\chi_I^N)(x)\Big]^{2q} \,.
\end{equation}
\end{remark}

\begin{proof} Since $N^{[k]}$ is integer-valued, without loss of generality we may assume $\lambda \ge 1/2$. We estimate
\begin{equation}\label{e.Nleveldecomp}
w(\{N^{[k]} > \lambda\})  \le \sum_{l\ge 0} w(\{2^l \lambda < N^{[k]} \le 2^{l+1}\lambda \})
\end{equation}
and it is not hard to see that
\begin{equation}\label{e.reducetoNl}
w(\{2^l\lambda < N^{[k]} \le 2^{l+1}\lambda\}) \le w(\{N^{[k]}_l>2^l \lambda\})
\end{equation}
$$\text{where } \qquad N^{[k]}_l := \sum_{T\in \T_l} 1_{2^k I_T}\,,$$
$$\text{and } \qquad \T_l := \{T\in \T: 2^k I_T \not\subset \{N^{[k]} > 2^{l+1}\lambda\} \}$$
Write $N_l$ for $N^{[0]}_l$, clearly $N_l\le N^{[k]}_l$ for $k\ge 0$. We first show that
\begin{equation}\label{e.boundNl}
\|N_l\|_\infty \le 2^{l+1}\lambda \,.
\end{equation}
Indeed, take any $x$, and let $\T_x= \{T\in \T_l: x\in I_T\}$. Clearly,
$$N_l(x) \le \sum_{T\in \T_x} 1_{2^k I_T} \,.$$
Since the collection of top intervals of elements of $T_x$ is nested, there is one minimal element. Note that if $I_1 \subset I_2$ are two intervals then for $k\ge 0$ we have $2^k I_1 \subset 2^k I_2$. Therefore the intervals $2^k I_T$ with $T\in \T_x$ are also nested and the minimal of them contains a point $y\in \{N^{[k]} \le 2^{l+1}\lambda\}$ by definition of $\T_l$. Therefore,
$$N_l(x) \le N^{[k]}(y) \le 2^{l+1}\lambda\,,$$
completing the proof of \eqref{e.boundNl}.

Now, denote $\P_l = \bigcup_{T\in \T_l}T$ and as usual
$$S_{\P_l}f = ( \sum_{P\in \P_l} |\<f,\phi_{P_1}\>|^2 \frac {1_{I_P}}{|I_P|})^{1/2} \,.$$
It follows from  \eqref{e.largesize}, \eqref{e.doubling}, and H\"older's inequality  that
\begin{equation}\label{e.presharp}
\alpha^{2q} \|N^{[k]}_l \|_{L^1(w)}  \le C 2^{\gamma k} \|S_{\P_l} f\|_{L^{2q}(w)}^{2q} \,.
\end{equation}
For $N$ large let $f_I=f\widetilde \chi_I$. The key estimate in our proof of \eqref{e.Nweak} is 
\begin{claim}\label{cl.M2bound} For any $s \in (0,1)$ and $\delta>0$ there is $C=C(\epsilon,s,N) < \infty$ such that
\begin{equation}\label{e.M2bound}
(S_{\P_l} f)^{\sharp} \le C \|N_l\|_\infty^{\delta}\Big( \M_2 f_I  + \big[\alpha  \M_2 (N_l^{\frac 1{2}}) \big]^{s} (\M_2 f_I)^{1-s}\Big) \,.
\end{equation}
\end{claim}
Below we show \eqref{e.Nweak} using the above claim. It follows from  \eqref{e.presharp}, \eqref{e.M2bound}, and the assumption $w\in A_q$ that
$$\alpha \|N^{[k]}_l\|_{L^1(w)}^{\frac 1{2q}} \le C 2^{O(k)} \|(S_{\P_l} f)^{\sharp} \|_{L^{2q}(w)}  $$
$$\le C 2^{O(k)}\|N_l\|_{\infty}^{\delta} \Big(\|f_I\|_{L^{2q}(w)} + \Big[\alpha  \|N_l^{1/2}\|_{L^{2q}(w)}\Big]^{s}\|f_I\|_{L^{2q}(w)}^{1-s}\Big)$$
$$\le C 2^{O(k)}\Big(\|N_l\|_{\infty}^{\delta}  \|f_I\|_{L^{2q}(w)}  +  \|N_l\|_\infty^{\delta + s(\frac 1 2 - \frac 1 {2q})}  \alpha^s\|N_l\|_{L^1(w)}^{\frac s{2q}} \|f_I\|_{L^{2q}(w)}^{1-s}\Big) \,.$$
Here $\delta>0$ and $s>0$ will be chosen very close to $0$. Consequently, after bootstrapping, it follows that for any $\epsilon>0$
$$\alpha\|N_l^{[k]}\|_{L^1(w)}^{\frac 1 {2q}} \le C  2^{O(k)}\|N_l\|_\infty^{\epsilon/2q} \|f_I\|_{L^{2q}(w)} \,.$$
Therefore, it follows from the bound $\|N_l\|_\infty \le 2^{l+1}\lambda$ of  \eqref{e.boundNl} that
$$ w(\{N_l^{[k]}>2^l\lambda\}) \le C  2^{O(k)} 2^{-l(1-\epsilon)} \alpha^{-2q} \lambda^{-1+\epsilon} w(I) [\inf_{x\in I} \M_{2q,w} f(x)]^{2q} $$
Choosing $\epsilon>0$ very small allows for summation  over $l\ge 0$ of the above estimate. Using \eqref{e.Nleveldecomp} and \eqref{e.reducetoNl}, we obtain the desired estimate \eqref{e.Nweak}. \\

\noindent \underline{Proof of Claim~\ref{cl.M2bound}:} 
Fix any dyadic $J$. For any $T\in \T_l$ let $T_J := \{P \in T: I_P \subset J\}$, and by decomposing $T_J$ into $O(1)$ subtrees we may assume that $T_J$ is a tree with a new top interval $I_T\cap J$ for every $T \in \T_l$. It suffices to show that for any $x\in J$:
\begin{equation}\label{e.localM2bound}
\frac 1 {|J|^{1/2}} \|(\sum_{T\in \T_l} |S_{T_J}f|^2)^{1/2}\|_{L^{2}} \le \text{the value at $x$ of RHS of \eqref{e.M2bound} .}
\end{equation}
By Lemma~\ref{l.separatedtrees}, for any $0<s \le 1$ there is $C=C_s <\infty$ such that
\begin{equation}\label{e.multitree} 
(\sum_{T\in \T_l}  \|S_{T_J} f\|_2^2)^{1/2}  \le C\|f\|_2 + C\alpha^{s} \|N_l^{1/2}\|_{L^2(J)}^s \|f\|_2^{1-s} \,.
\end{equation}
Here we've used the fact that for any $P\in \P$:
$$\frac{|a_P|}{|I_P|^{1/2}} = \Big(\frac 1 {w(I_P)} \int |a_P|^2 \frac{1_{I_P}}{|I_P|} w(x)dx \Big)^{1/2}  \le \alpha \,.$$
Since for any $P\in T_J$ we have $I_P\subset I\cap J$, it follows from Corollary~\ref{c.BMO} that
\begin{equation}\label{e.singletree}
\|S_{T_J}f\|_{BMO}  \le C \inf_{x\in I\cap J} \M_1 (f \widetilde \chi_{I\cap J}^{N}) (x) \,.
\end{equation}
Interpolate the estimates \eqref{e.multitree} and \eqref{e.singletree} to prove \eqref{e.localM2bound} using a now-standard localization argument (see e.g. \cite{lacey-thiele-bht}). The idea is to
 decompose $f=\sum_{k\ge 0} f_k$ where $f_0 = f1_{I\cap J}$ and $f_k= f1_{2^k (I\cap J)\setminus 2^{k-1}(I\cap J)}$ for $k\ge 1$ and apply \eqref{e.multitree} and \eqref{e.singletree} to $f_k$. More specifically, for $p \in (2,\infty)$ we have
$$\|(\sum_{T\in \T_l} |S_{T_J}f_k|^p)^{1/p} \|_p = \Big(\sum_{T\in \T_l} \|S_{T_J}f_k\|_p^p\Big)^{\frac 1 p}$$
$$ \le  \Big(\sum_{T\in \T_l} \|S_{T_J}f_k\|_2^2\Big)^{\frac 1 p}  \sup_{T\in \T_l}\|S_{T_J}f_k\|_{BMO}^{1-\frac 2 p} $$
$$ \le C_{N,p} 2^{-Nk} |I\cap J|^{1/p} \inf_{x\in I\cap J} \Big( \M_2 f_I(x) +  \big[\alpha \M_2(N_l^{1/2})(x)\big]^{\frac {2s}p}  \big[\M_2 f_I(x) \big]^{1-\frac {2s}p}\Big)\,.$$
Summing over $k\ge 0$ we obtain
$$\|(\sum_{T\in \T_l} |S_{T_J}f|^p)^{1/p} \|_p$$
\begin{equation}\label{e.Lplp}
\le C  |J|^{1/p} \inf_{x\in J} \Big( \M_2 f_I(x) +  [\alpha \M_2(N_l^{1/2})(x)]^{\frac {2s}p}  \M_2 f_I(x)^{1-\frac {2s}p}\Big)  \,.
\end{equation}
On the other hand, using H\"older's inequality it follows that
\begin{equation}\label{e.Jholder}
\|(\sum_{T\in \T_l} |S_{T_J}f_k|^2)^{1/2} \|_p  \le  \|N_l\|_\infty^{\frac 1 2 - \frac 1 p} \|(\sum_{T\in \T_l} |S_{T_J}f|^p)^{1/p} \|_p \,.
\end{equation}
Combining \eqref{e.Lplp} and \eqref{e.Jholder} and use H\"older, it follows that
$$\frac 1 {|J|^{1/2}} \|(\sum_{T\in \T_l} |S_{T_J}f|^2)^{1/2}\|_{L^{2}} \le \frac 1 {|J|^{1/p}} \|(\sum_{T\in \T_l} |S_{T_J}f|^2)^{1/2}\|_{L^{p}}$$
$$\le C \|N_l\|_\infty^{\frac 12 - \frac 1 p} [ \M_2 f_I(x) + \big[\alpha \M_2 (\sqrt N_l)(x) \big]^{\frac{2s}p} \big[\M_2 f_I (x) \big]^{1-\frac {2s}p}  \,.$$
Choosing $p>2$ sufficiently close to $2$ we obtain the desired estimate \eqref{e.localM2bound}.

\end{proof}

The following Lemma,  needed for our proof of Claim~\ref{cl.M2bound}, is contained implicitly in \cite{ThieleHabilitation}, where in fact a stronger logarithmic variant was proved (see also \cite{hytonen-lacey} for a vector valued generalization).

\begin{lemma} \label{l.separatedtrees} Let $\T$ be a well-separated collection of $2$-overlapping trees and let $\P = \bigcup_{T\in \T} T$. Then for any $0<s\le 1$ it holds that
$$(\sum_{P\in \P} |\<f,\phi_{P_1}\>|^2)^{1/2}$$
\begin{equation}\label{e.separatedtrees}
\le C_s \Big(\|f\|_2 +    \Big[\sup_{P\in \P} \frac{|\<f,\phi_{P_1}\>|}{|I_P|^{1/2}} (\sum_{T\in \T} |I_T|)^{1/2}\Big]^s \|f\|_2^{1-s}  \Big)\,.
\end{equation}
\end{lemma}

Remark: While any $0<s<1$ would be   enough for applications to the Lebesgue setting of Carleson theorems (see e.g. \cite{lacey-thiele-carleson} and \cite{oberlin-et-al} where $s=1/3$ is used), our applications to Claim~\ref{cl.M2bound} require arbitrarily small $s>0$. We include a proof of \eqref{e.separatedtrees} (following largely \cite{ThieleHabilitation}) below. 

\begin{proof} Without loss of generality we may assume $\|f\|_2=1$. Denote 
$$N=\sum_{T\in \T}1_{I_T} \qquad ,  \qquad a_P = \<f,\phi_{P_1}\>\,, $$
$$A=(\sum_{P\in \P}|a_P|^2)^{1/2} \ , \qquad   B=\sup_{P\in \P} \frac{|a_P|}{|I_P|^{1/2}} \ \  .$$
We then divide $\P$ into subcollections $\P_k$, where for any $k\ge 0$ we have
$$\P_k = \{P \in \P: 2^{-k-1}B< \frac{|a_P|}{|I_P|^{1/2}} \le 2^{-k} B\}, 
$$
and let $\P_{\ge k} = \bigcup_{j\ge k} \P_j$. Using the known special case $s=1/3$ of \eqref{e.separatedtrees} proved in \cite{lacey-thiele-carleson} (see also \cite{oberlin-et-al} for a setting similar to the current paper) for the restriction to $\P_{\ge k}$ of the tree collection $\T$,  we have
$$(\sum_{P\in \P_{\ge k}} |a_P|^2)^{1/2} \le C  + C (2^{-k}B)^{1/3}\|N\|_1^{1/6} $$
in particular for $k \ge    \max(0,\log_2 \big[  B (\sum_{T\in \T}|I_T|)^{1/2} \big])$ we have
\begin{equation}\label{e.largek}
(\sum_{P\in \P_{\ge k}} |a_P|^2)^{1/2} \le C \,.
\end{equation}
On the other hand, it follows from the definition of $\P_k$ that
\begin{equation}\label{e.pk1}
(\sum_{P\in \P_k} |a_P|^2)^{1/2} \sim 2^{-k}B (\sum_{P\in \P_k}|I_P|)^{1/2} \,.
\end{equation}
We can also view $\P_k$ as a collection of single-bitile trees, which is clearly well-separated. Thus again using the known case $s=1/3$ of \eqref{e.separatedtrees}, it follows that
\begin{equation}\label{e.pk2}
(\sum_{P\in \P_k} |a_P|^2)^{1/2} \le C + C \Big[2^{-k}B (\sum_{P\in \P_k}|I_P|)^{1/2}\Big]^{1/3}\,.
\end{equation}
Combining  \eqref{e.pk1} and \eqref{e.pk2}, it follows that for any $k\ge 0$ we have
\begin{equation}\label{e.smallk}
(\sum_{P\in \P_k} |a_P|^2)^{1/2} \le C\,.
\end{equation}
Using \eqref{e.largek}   and \eqref{e.smallk} we obtain
$$\sum_{P\in \P} |a_P|^2 \le C+ C\max(0,\log_2   \big[B \|N\|_1^{1/2} \big])\,.$$
Using the trivial estimate $\max(0,\log x) \le x$ for any $x>0$, we obtain
$$(\sum_{P\in \P} |a_P|^2)^{1/2} \le C (1+  \big[B\|N\|_1^{1/2} \big]^s) $$
for any $0<s\le 1$, as desired.
\end{proof}

\subsection{Decomposition by density}
Since $|g|\le 1_G$, the density of any collection is bounded above by $1$. For the result below, it is important that the constant $D$ in the definition of $density$ is sufficiently large, much bigger than the doubling exponent $\gamma$ of $w$. We  return to this point in the proof.

\begin{lemma}\label{l.density} For any collection $\P$ of bitiles and any $\alpha>0$ we can find a collection $\T$ of trees such that
the density of $\P - \bigcup_{T\in \T} T$ is bounded above by $\alpha$ and
$$\sum_{T\in \T} w(I_T) \le C\alpha^{-r'} w(G)$$
here $r$ is the variational exponent used in the definition of density.
\end{lemma}

Remark: This is a weighted extension of  \cite{oberlin-et-al}*{Proposition 4.4}, and the proof below is adapted from \cite{oberlin-et-al}, which is in turn a variational adaptation of the standard argument. The variant of Lemma~\ref{l.density} with improved density follows immediately, since for any $\P$ we have $\widetilde{density}(\P) \le C\, density(\P)$.

\begin{proof} If $\density(\P) > \alpha$ then there is a nonempty tree $T\subset \P$ such that
\begin{equation}\label{e.largedensity}
\omega(I_T) \le \alpha^{-r'} \int \widetilde \chi_{I_T}^D|g|^{r'} \sum_{j: N_j\in \omega_T} |d_j|^{r'} w \,.
\end{equation}
We select $T$ such that $|I_T|$ is maximal, and then by enlarging $T$ (keeping $I_T$ and $\xi_T$) if necessary we may assume  that $T$ is maximal in $\P$ with respect to set inclusion. Let $T_+$ and $T_-$ be the maximal trees in $\P$ with the same top interval as $T$ but with top frequencies $\xi_T - \frac 1{2|I_T|}$ and $\xi_T+\frac 1{2|I_T|}$ respectively. We then remove from $\P$ the union of $T, T_+, T_-$. Continuing this selection process, which will stop since $\P$ is assumed finite, we obtain a collection $\T$ of trees, such that 
$$\density(\P - \bigcup_{T\in \T} (T \cup T_- \cup T_+))\le \alpha \,.$$
It remains to show that
$$\sum_{T\in \T}w(I_T) \le C \alpha^{-r'}w(G) \,. $$
By the selection algorithm, it is not hard to see that for $T \ne T'$ in $\T$ the rectangles $I_T \times \omega_T$ and $I_{T'} \times \omega_{T'}$ are disjoint. Now, it follows from \eqref{e.largedensity} that  for any $T\in \T$ there exists an integer $k =k(T)\ge 0$ such that
\begin{equation}\label{e.kdecay}
\omega(I_T) \le C 2^{-Dk} \alpha^{-r'} \int_{2^{k}I_T} |g|^{r'} \sum_{j: N_j \in \omega_T} |d_j|^{r'} w \,.
\end{equation}
We then sort the trees in $\T$ according to the value of $k(T)$. More specifically for each $k\ge 0$ let $\T_k = \{T\in \T: k(T) = k\}$. It suffices to show that
\begin{equation}\label{e.topintervalk}
\sum_{T\in \T_k} w(I_T) \le C \alpha^{-r'} 2^{-k} w(G) \,.
\end{equation}
Fix $k$.  Select a subcollection $\S_k \subset \T_k$ such that the rectangles $2^k I_S \times \omega_S$ with $S \in \S_k$ are pairwise disjoint, and such that
\begin{equation}\label{e.overlap}
\sum_{T\in \T_k} \omega(I_T) \le C\sum_{S\in \S_k} \omega(2^{k+2} I_S) \,.
\end{equation}
Note that this will imply the desired estimate \eqref{e.topintervalk}.   
By choosing $D>\gamma+10$, where $\gamma$ is the doubling exponent for $w$, it follows from \eqref{e.kdecay} and \eqref{e.overlap} that
$$\sum_{T\in \T_k} \omega(I_T) \le C 2^{k\gamma} \sum_{S\in \S_k} \omega(I_S)$$
$$\le C 2^{-k} \alpha^{-r'} \int \sum_{j} \sum_{S\in \S_k} 1_{\{(x,N_j(x))\in 2^k I_S \times \omega_S \}} |d_j|^{r'} |g|^{r'} w$$
$$\le  C 2^{-k} \alpha^{-r'} \int  |g|^{r'}\sum_{j} |d_j|^{r'} w \le C 2^{-k} w(G) \,.$$
It remains to select $\S_k$. Assuming without loss of generality that $\T_k$ is nonempty. Then we choose $S\in \T_k$ such that $|I_S|$ is maximal and then remove all $T\in \T$ if 
$$2^k I_T\times \omega_T  \cap 2^k I_S \times \omega_S \ne \emptyset \,.$$ 
Starting from the remaining collection, we repeat the above selection procedure  until no trees are left. We then let $\S_k$ be the collection of selected trees. For any $S\in \S_k$, let $\T_S$ denote the collection of trees in $\T$ that are removed after $S$ is selected, then to show \eqref{e.overlap}  it suffices to show that
\begin{equation}\label{e.countingoverlap}
\sum_{T \in \T_S} 1_{I_T} \le C 1_{2^{k+2}I_S} \,.
\end{equation}
Note that if $T\in \T_S$ then $|I_T|\le |I_S|$ and $2^k I_T\cap 2^k I_S \ne\emptyset$, so clearly $I_T\subset 2^{k+2} I_S$. Also $|\omega_T| \ge |\omega_S|$ and $\omega_T \cap \omega_S \ne \emptyset$, so out of any four trees in $\T_S$ at least two of them will have overlapping top  frequency intervals. The desired estimate \eqref{e.countingoverlap} then follows from the fact that the rectangles $I_T\times \omega_T$ (with $T\in \T_S$) are disjoint. 
\end{proof}

\section{The tree estimate}

In this section we prove several estimates for the restriction of the (model) Carleson operator to a tree. Lemma~\ref{l.tree-est} is applicable to any tree, while Lemma~\ref{l.improved-tree-est} improves the $L^1$ case of Lemma~\ref{l.tree-est} when the elements of the underlying tree are disjoint in the phase plane.

Recall that for any bitile collection $\Q$ we denote
$$C_{\Q} f (x) = \sum_{P\in \Q} \<f,\phi_{P_1}\> \phi_{P_1}(x) d_P(x)$$
with $d_P$ defined as follows: First, $(d_k)_{k\ge 1}$ and $N_k$ are two sequences of measurable functions of $x$, such that
\begin{itemize}
\item For each $x$ there is some integer $K=K(x) <\infty$ such that $d_k(x) = 0$ for $k>K$, and uniform over $x$ we have $\sum_{k\ge 0} |d_k(x)|^{r'} = 1$.
\item For any $x$ we have $N_0(x)<N_1(x)< \dots $.
\end{itemize}
Then for each $x$ define $d_P(x)=0$ unless  there exists an index $k$ such that $N_{k-1} \not\in  \omega_{P}$ and $N_k \in  \omega_{P_2}$, in which case such index is unique and we define $d_P(x):=d_k(x)$. We note that if $P \in \P$ then 
$$\int \widetilde \chi_{I_P}^D |g|^{r'} \sum_{k: N_k \in  \omega_{P_2}} |d_k|^{r'}w \le C w(I_P)\density(\{P\})^r  \,.$$
The above observation will be used implicitly below.
\begin{lemma}\label{l.tree-est}
Let $T$ be a tree.  Assume $s\in [1,r']$.  Then there exists some $C=C(s,w)<\infty$ such that
\begin{equation}\label{e.tree-est-core}
\|1_{I_T} gC_{T} f \|_{L^{s}(w)}   \le C  w(I_T)^{1/s} \size(T) \density(T)\,,
\end{equation}
and furthermore for any $N> 0$ there exists $C=C(N,s,w)<\infty$ such that the following inequality holds for any $k\ge 0$:
$$\|1_{2^{k+1} I_T\setminus 2^k I_T} gC_{T} f \|_{L^{s}(w)}  $$ 
\begin{equation}\label{e.tree-est-tail} 
\le C 2^{-Nk} w(I_T)^{1/s} \size(T) \density(T) \,.
\end{equation}
\end{lemma}
Remark: As a consequence, we obtain for any $s\in [1,r']$:
\begin{equation}\label{e.tree-est}
\|gC_{T} f \|_{L^{s}(w)}   \le C  w(I_T)^{1/s} \size(T) \density(T) \,.
\end{equation}
\begin{proof} By H\"older's inequality and using the doubling property of $w$ it suffices to show \eqref{e.tree-est-core} and \eqref{e.tree-est-tail}  for $s=r'$, and this will be assumed in the rest of the proof. By dividing $ T$ into two subtrees,  if necessary,   we can assume that the tree is either $2$-overlapping or $2$-lacunary. We will return to this distinction below. \\

\noindent \underline{Proof  of \eqref{e.tree-est-core}:} We will prove a stronger estimate, where the restriction $1_{I_T}$ is not required. Let $\J$ be the set of maximal dyadic intervals such that
$$I_P \not\subset 3J$$
for any  $P\in T$. It is not hard to see that $\J$ partitions $\R$. Let
\begin{equation}\label{e.defTJ}
T_J := \{P\in T: |I_P|\le C_4 |J|\}\,,
\end{equation}
some absolute constant $C_4\ge 4$ to be chosen later.  The left hand side of \eqref{e.tree-est-core} (with $s=r'$ now) is bounded above by $A+B$ where
\begin{eqnarray}
\label{e.Jerr}  A&:=& \Big(\sum_{J \in \J} \int_J |gC_{T_J}f |^{r'} w \Big)^{1/r'}\\
\label{e.Jmain} B&:=& \Big(\sum_{J \in \J} \int_J |g C_{T\setminus T_J}f) |^{r'} w \Big)^{1/r'} \,.
\end{eqnarray}
To bound $A$, we fix $J\in \J$ and first estimate the contribution of each $P\in T_J$:
$$\Big(\int_J |g C_{\{P\}}f|^{r'} w\Big)^{1/r'} \le C_{N} \frac{|\<f,\phi_{P_1}\>|}{|I_P|^{1/2}} (\int |1_J g \widetilde \chi_{I_P}^{N+D} d_P|^{r'} w)^{1/r'}$$
$$\le C w(I_P)^{1/r'} \size(\{P\}) \density(\{P\}) \sup_{y\in J} \widetilde \chi_{I_P}(y)^N \,.$$
Using the triangle inequality, it follows that
$$\Big(\int_J |g C_{T_J}f|^{r'} w\Big)^{1/r'}$$ 
\begin{equation}\label{e.contributionJ}
 \le C \size(T) \density(T)  \sum_{P\in T_J} w(I_P)^{1/r'}  (1+\frac{\dist(J,I_P)}{|I_P|})^{-4N} \,.
\end{equation}
By the $A_\infty$ property of $w$ there exists constants $\beta_0>0$ such that if $I\subset I'$ are two intervals then
$$\frac{w(I)}{w(I')} \le C(\frac{|I|}{|I'|})^{\beta_0} \,.$$
Without loss of generality, we may choose the doubling constant $\gamma$ in \eqref{e.doubling} to be large enough such that $\gamma > \beta_0$.
 
For any $P\in T_J$ we can find an interval $K$ of length comparable to $|I_P| + |J| + \dist(J,I_P)$ that contains both $I_P$ and $J$. Since $|I_P| =O(|J|)$ we can choose $K$ to be a dilation of $J$. We then have
$$\frac{w(I_P)}{w(J)} = \frac{w(I_P)}{w(K)} \frac {w(K)}{w(J)} \le C (\frac{|I_P|}{|K|})^{\beta_0} (\frac{|K|}{|J|})^{\gamma}$$
$$= C (\frac{|I_P|}{|J|})^{\beta_0} (\frac{|K|}{|J|})^{\gamma-\beta_0} \le C(\frac{|I_P|}{|J|})^{\beta_0} (1+ \frac{\dist(J,I_P)}{|J|})^{\gamma-\beta_0}$$
$$\le C(\frac{|I_P|}{|J|})^{\beta_0} (1+ \frac{\dist(J,I_P)}{|I_P|})^{\gamma-\beta_0} \,.$$
Therefore by choosing $N$ sufficiently large it follows from \eqref{e.contributionJ} that
$$\Big(\int_J |g C_{T_J}f|^{r'} w\Big)^{1/r'}$$ 
$$\le C \size(T) \density(T)  \sum_{P\in T_J} (\frac{|I_P|}{|J|})^{\beta_0/r'} w(J)^{1/r'}  (1+\frac{\dist(J,I_P)}{|I_P|})^{-3N}$$
$$= C \size(T) \density(T)    w(J)^{1/r'} \sum_{k\ge -1} \sum_{|I_P|=2^{-k}|J|} 2^{-k\beta_0/r'}(1+\frac{\dist(J,I_P)}{|I_P|})^{-3N} \ .$$
Using the fact that $3J$ does not contain any $I_P$, $P\in T_J$, and the fact that elements of $T_J$ of the same size are spatially disjoint, it is not hard to bound the last display by
$$\le C\size(T) \density(T)    w(J)^{1/r'} (1+\frac{\dist(J,I_T)}{|I_T|})^{-2N}$$

Thus, we can bound $A$ by
$$A \le C \size(T) \density(T) \Big(\sum_{J\in \J} w(J)(1+\frac{\dist(J,I_T)}{|I_T|})^{-2Nr'}\Big)^{1/r'}\,.$$
Note that by definition $3J$ does not contain $I_T$. It follows that for any $x\in J$
$$1+\frac{\dist(J,I_T)}{|I_T|} \sim  1+\frac{|x-c(I_T)|}{|I_T|} \,.$$
Choosing $N$ large and using disjointness of $J$'s, we obtain
$$\sum_{J\in \J} w(J)(1+\frac{\dist(J,I_T)}{|I_T|})^{-2Nr'} \le C\int \widetilde\chi_{I_T}^{N} w \le Cw(I_T) \,.$$
Consequently, we have
$$A \le C w(I_T)^{1/r'} \size(T)\density(T) \,.$$

To bound $B$,  let $F_J = \bigcup_{T\in T\setminus T_J} \omega_{P_2}$, we first show that
\begin{equation}\label{e.densityJlevel}
\int_J |g|^{r'} \sum_{j: N_j\in F_J}|d_j|^{r'}w \le C w(J) [\density(T)]^{r'} \,.
\end{equation}
\begin{proof}[Proof of \eqref{e.densityJlevel}] We  construct $O(1)$ non-empty subtrees of $T$ such that $F_J$ is contained inside the union of the frequency intervals of these trees. The top interval of each such subtree will be of length $\sim |J|$ and will be contained in some $O(1)$ dilation of $J$.  Clearly, \eqref{e.densityJlevel} follows as a consequence of this construction.

To construct these trees, first we construct their (common) top interval $J_0$. Let $\pi(J)$ be the dyadic parent of $J$. Then we can find $Q\in T$ such that $I_{Q}\subset 3\pi(J)$, therefore we can select a dyadic interval $J_0$ such that
$$I_{Q} \subset J_0 \subset 3\pi(J)\,,  \ \  |J_0|\ge |J| \,.$$
Now, note that by dividing $T$ into three trees if necessary, we may assume without loss of generality that only one of the following scenarios happens:
\begin{itemize}
\item[(i)] $\xi_T \in \omega_{P_2}$  for {every} $P\in T$, or
\item[(ii)] $\xi_T < \inf \omega_{P_2}$  for {every} $P\in T$, or
\item[(iii)]  $\xi_T \ge \sup \omega_{P_2}$  for {every} $P\in T$.
\end{itemize}
In each of these scenarios, one tree will be constructed.  The desired tree has only one element $Q$ and has top data $(J_0, \omega_0)$, and $\omega_0$ is constructed below: it will be shown that
\begin{equation}\label{e.frequencyinclusion}
F_J \subset \omega_0 \subset \widetilde \omega_Q \,.
\end{equation}
We note that by choosing $C_4$ large in the definition \eqref{e.defTJ} we can ensure that for any $P\in T\setminus T_J$ we have $|\omega_P| < 1/|J_0|$. Furthermore, if $C_2>1$ we can also ensure that $|\omega_P|<\frac{C_2-1}{2}|J_0|$.

If (i) is satisfied, we let $\omega_0$ be the dyadic interval of length $1/|J_0|$ containing $\xi_T$. It is clear that  for any $P\in T\setminus T_J$ we have $\omega_{P_2}\subset \omega_0$ and $\omega_0\subset \omega_{Q_2}$, and \eqref{e.frequencyinclusion} follows immediately. 

If (ii) is satisfied, we let $\omega_0=[\xi_T, \xi_T + \frac 1{|J_0|})$. Since for any $P\in T\setminus T_J$ we have $|\omega_P|<|\omega_0|<|\omega_{Q_2}|$, it follows that we always have $\omega_{P_2}\subset \omega_0 \subset \widetilde \omega_Q$, as desired.

If (iii) is satisfied, we let $\omega_0=[\xi_T -\frac 1{|J_0|}, \xi_T)$, and argued as in situation (ii).

This completes the proof of \eqref{e.densityJlevel}.  
\end{proof}

Below we return to our task of estimating $B$. We remark that any $J\in \J$ that contributes to $B$ must satisfies $|J|<|I_T|/C_4\le |I_T|/4$, therefore  $J\subset 3I_T$. We now consider two cases:\\

\noindent \underline{Case 1: $T$ is $2$-lacunary:} By ensuring that the constant $K_0$ in the separation assumption (S\ref{i.S3}) is sufficiently large,  it follows that for $P,P'\in T$ with $|I_P|>|I_{P'}|$ we have $\omega_{P_2}\subset \omega_{P'}$. Using the fact that  $\{N_j(x)\}$ is an increasing sequence for every $x$, it follows from a geometrical consideration that for each $x$ there is at most one $m$ and such that $d_P(x)\ne 0$ for some $P\in T$ with $|I_P|=2^m$. Here it is important that the limiting condition reads $\{N_{j-1}\not\in \omega_P, N_j\in  \omega_{P_2}\}$.  Now, uniformly over $m$ we have
$$\sum_{P \in T: |I_P|=2^m} (1+\frac{|x-c(I_P)|}{|I_P|})^{-2} = O(1) \,.$$
It then follows from \eqref{e.densityJlevel} that
$$\|1_J g C_{T\setminus T_J}f\|_{L^{r'}(w)} \le C\sup_{P\in T}\frac{|\<f,\phi_{P_1}\>|}{|I_P|^{1/2}} (\int_J |g|^{r'} \sup_{k: N_k\in \omega_{T_0}}|d_j|^{r'}w)^{1/r'}$$
$$\le C\size(T)\density(T)w(J)^{1/r'}.$$
Consequently we obtain the desired estimate:
\begin{align*}
B \le C \Big(\sum_{J\in \J} w(J)\Big)^{1/r'} \size(T)\density(T) 
\\&\le Cw(I_T)^{1/r'} \size(T)\density(T) \,.
\end{align*}

\noindent \underline{Case 2: $T$ is $2$-overlapping:} We  estimate pointwise
$$|C_{T\setminus T_J} f(x)|$$ 
$$\le \Big(\sum_{j} |\sum_{P\in T\setminus T_J: N_{j-1}\notin \omega_P, N_j \in  \omega_{P_2}} \<f,\phi_{P_1}\>\phi_{P_1}|^r\Big)^{1/r}\Big(\sum_{j: N_j \in \omega_{T_0}} |d_k(x)|^{r'}\Big)^{1/r'}$$
Therefore
\begin{equation}\label{e.collectfreq}
\|1_J g C_{T\setminus T_J}f\|_{L^{r'}(w)}
\end{equation}
$$\le Cw(J)^{1/r'} \density(T) \sup_{x\in J}\Big(\sum_{j} |\sum_{P\in T\setminus T_J: N_{j-1}\notin  \omega_P, N_j \in  \omega_{P_2}} \<f,\phi_{P_1}\>\phi_{P_1}|^r\Big)^{1/r} \,.$$
Note that for any $P$ the frequency support of $\phi_{P_1}$ is contained inside $C_3\omega_{P_1}=(1-c) C_2\omega_{P_1}$ for  $c=1-\frac {C_3}{C_2}\in (0,1)$ which is uniform over $P$'s. Recall that $T$ is a $2$-overlapping tree and  the relative position of the tiles in each bitile are uniform over $\P$. 

Now, by choosing the constant $K$ in the separation assumption (S\ref{i.S3}) to be sufficiently large, we can find a lacunary family of smooth Littlewood-Paley projection operators $\Pi_n$ such that: $\Pi_n$ is a smooth Fourier multiplier operator whose symbol is supported in  $\{|\xi| =O(2^n)\}$, and furthermore (thanks to separation) $\Pi_n \Pi_k = \Pi_k$ for any $n < k$ and $\phi_{P_1} = (\Pi_n - \Pi_{n-1})\phi_{P_1}$ for $n=\log_2 |I_P|$. 

It follows that for any $x\in J$ we can bound
$$\Big(\sum_{k} |\sum_{P\in T\setminus T_J: N_{k-1}\notin  \omega_P, N_k \in  \omega_{P_2}} \<f,\phi_{P_1}\>\phi_{P_1}|^r\Big)^{1/r}$$
$$\le \sup_{K, n_0< \dots < n_K < O(1)-\log_2 |J|} (\sum_{j=1}^K |(\Pi_{n_j} -\Pi_{n_{j-1}})g_T|^{r})^{1/r}\,.$$
where $g_T:=\sum_{P\in T} \<f,\phi_{P_1}\>\phi_{P_1}$. The last display can be rewritten as
$$=\sup_{K, n_0< \dots < n_K< O(1) - \log_2 |J|} (\sum_{j=1}^K |\Pi_{\log_2|J|}(\Pi_{n_j} -\Pi_{n_{j-1}}) g_T |^{r})^{1/r}$$
$$\le M_J (\sup_{K,n_0< \dots < n_K} (\sum_{j=1}^K |\Pi_{n_j} g_T-\Pi_{n_{j-1}}g_T|^{r})^{1/r})\,,$$
using Minkowski's inequality and standard arguments. Here, $M_J$ denotes the following local maximal operator:
$$M_J f =\sup_{I: J\subset I} \frac 1 {|I|}\int_{I} |f| \,.$$
For simplicity we denote by $\|g_T\|_{V^r}$ the variational expression inside $M_J$ in the above estimate. Recall that all the $J$ such that $T\setminus T_J$ are disjoint and contained in $3I_T$. Thus, it follows from \eqref{e.collectfreq} and the above estimate that
$$B \le C\density(T)\Big(\sum_{J\in \J} w(J) M_J (\|g_T\|_{V^r})^{r'}\Big)^{1/r'}$$
$$\le C\density(T) \|1_{3I_T}M(\|g_T\|_{V^r})\|_{L^{r'}(w)}$$
$$\le C\density(T) w(I_T)^{1/r' - 1/(2q)} \|M(\|g_T\|_{V^r})\|_{L^{2q}(w)}\,,$$
since $r'<2<2q$. Using $w\in A_q\subset A_{2q}$ and Lemma~\ref{l.weightLepingle} we obtain
$$B\le C\density(T) w(I_T)^{1/r' - 1/(2q)} \|g_T\|_{L^{2q}(w)}\,.$$
To show the desired bound for $B$ it remains to show that
$$\|g_T\|_{L^{2q}(w)} \le C w(I_T)^{1/(2q)} \size(T) \,.$$
Take $h$ to be any function in $L^{(2q)'}(w)$ where $(2q)'$ denote the dual exponent of $2q$. Let $\sigma=w^{-\frac{(2q)'}{2q}}$, since $w\in A_q\subset A_{2q}$ it is clear that $\sigma \in A_{(2q)'}$. We have
$$\<g_T, w h\>= \sum_{P\in T} \<f,\phi_{P_1}\> \<hw,\phi_{P_1}\>$$
$$\le \int (\sum_{P\in T} |\<f,\phi_{P_1}\>|^2 \frac{1_{I_P}}{|I_P|})^{1/2} (\sum_{P\in T} |\<hw,\phi_{P_1}\>|^2 \frac{1_{I_P}}{|I_P|})^{1/2}dx$$
$$\le \|S_T f\|_{L^{2q}(w)} \|S_T(hw)\|_{L^{(2q)'}(\sigma)} \,.$$
Then using the John-Nirenberg characterization of size in Lemma~\ref{l.BMO} and the estimate \eqref{e.M1bound}, it is not hard to see that
$$\<g_T, w h\> \le C w(I_T)^{1/(2q)} \size(T) \|hw\|_{L^{(2q)'}(\sigma)}$$
$$= Cw(I_T)^{1/(2q)} \size(T) \|h\|_{L^{(2q)'}(w)}\,,$$
as desired.\\

\noindent \underline{Proof of \eqref{e.tree-est-tail}:} Let $\widetilde g = g1_{2^{k+1}I_T\setminus 2^k I_T}$. Note that it suffices to consider $k\ge 2$. One proceeds as in the above proof of \eqref{e.tree-est-core} with $\widetilde g$ in place of $g$. It suffices to observe that in the above proof of \eqref{e.tree-est-core} we don't need to consider \eqref{e.Jmain} for $k\ge 2$ since all the $J$ that contributes to this term is contained inside $3I_T$. Furthermore, any $J$ that contributes to \eqref{e.Jerr} satisfies
$$\frac{\dist(J, I_T)} {|I_T|} \ge C 2^k\,,$$
therefore in the rest of the proof one could easily introduce a decaying factor.
\end{proof}

\begin{lemma}\label{l.improved-tree-est}
Let $T$ be a tree and suppose that any two bitiles of $T$ are disjoint.  Then there exists some $C=C(w)<\infty$ such that
\begin{equation}\label{e.improved-tree-est}
\|gC_{T} f \|_{L^1(w)}   \le C  w(I_T) \size(T) \widetilde{\density}(T) \,.
\end{equation}
\end{lemma}
\begin{proof} Clearly the elements of $T$ must be spatially disjoint using the separation assumption on $\P$ and the fact that $T$ is a tree. Thus, by the triangle inequality it suffices to show \eqref{e.improved-tree-est} for any single-element tree, but the improved $L^1$ tree estimate is clear for these trees.
\end{proof}

\section{Weighted variational inequalities for Littlewood--Paley families}
In this section, we prove weighted extensions of a L\'epingle inequality, namely a 
variational inequality for for Littlewood--Paley families \cites{lepingle, bourgain, jones-seeger-wright, pisier-xu}.  Note that the dyadic variant of Lemma~\ref{l.weightLepingle} below was proved in \cite{weightWalsh}.

\begin{definition}\label{l.littlewoodpaley}
Fix an absolute constants $C \in (1,\infty)$, and $ \{C_N \;:\; N\in \mathbb N \}$, $ m\ge 1$.  A sequence of functions  $(f_{j})_{j\in \Z}$  
is a \emph{Littlewood--Paley family}  each each $f_{j}$ has frequency support inside $\{\frac 1 C 2^{-j} < |\xi| < C 2^{-j}\}$, and
\begin{equation*}
\lvert  \frac {d ^{N}}  {dx ^{N}} f_I (x) \rvert  \le C_N  2 ^{-j  N} [1 + \lvert  x\rvert 2 ^{-j}] ^{-m}  
\end{equation*}
\end{definition}
 
\begin{lemma}\label{l.weightLepingle} Let $1<p<\infty$, $w\in A_p$ and $r\ne 2$. Let $s=\min(r,2)$. Then for any Littlewood--Paley family $(f_ j )$ we have
\begin{equation}\label{e.rnot2}
\|\sup_{K, N_0< \dots<N_K}  (\sum_{k=1}^K |\sum_{N_{k-1}< j \le N_k} f_j|^r )^{1/r}\|_{L^p(w)} \le C \|(\sum_{j} |f_j|^{s})^{1/s}\|_{L^p(w)} \,.
\end{equation}
\end{lemma}

\begin{proof} Let $\Delta_j$ be Littlewood--Paley projection of $f$ into an enlarged frequency range $\{\frac 1{2C} 2^{-j} < |\xi| < 2C 2^{-j}\}$, such that $\Delta_j f_j = f_j$. It then suffices to show that for any $w\in A_p$  and any family of Littlewood--Paley projections $(\Delta_j)$ and any vector valued function $\f = (f_j)_{j\in \Z}$ we have
\begin{equation}\label{e.lpfam}
\|\sup_{K, N_0< \dots<N_K}  (\sum_{k=1}^K |\sum_{N_{k-1}< j \le N_k} \Delta_j f_j|^r )^{1/r}\|_{L^p(w)} \le C \|(\sum_{j} |f_j|^s)^{1/s}\|_{L^p(w)} \,.
\end{equation}
Let $T \f$ denote the variational operator inside $\|.\|_{L^p(w)}$ in the left hand side of \eqref{e.lpfam}. Then it suffices to show the following pointwise bound for the dyadic sharp maximal function of $T\f$: for any $1<t<\infty$,
\begin{equation}\label{e.Tfsharp}
(T\f)^{\sharp}(x) \le \M_t (\f)(x)\,, \ \ |\f| = (|\sum_j|f_j|^s)^{1/s}\,,
\end{equation}
Indeed, since $w\in A_p$ this will imply that
$$\|T\f\|_{L^p(w)} \le C\|(T\f)^{\sharp}\|_{L^p(w)} \le C\|\M_t (\f)\|_{L^p(w)} \,.$$
We now take $1<t<p$ sufficiently small such that $w\in A_{p/t}$, and the desired estimate \eqref{e.lpfam} then follows:
$$\|T\f\|_{L^p(w)} \le C\|\M_t (\f)\|_{L^p(w)} \le C\|\f \|_{L^p(w)}  \,.$$

It remains to show \eqref{e.Tfsharp}, and we use an argument from \cite{do-muscalu-thiele}. Take any dyadic interval $I$ containing $x$. Let $c_j$ be a constant defined as follows:
$$c_j = \begin{cases}\frac 1 {|I|} \int \phi_j * f_j , & 2^j < \frac 1{|I|} \\
0, & \text{otherwise}\end{cases}$$
where $\phi_j$ is the corresponding convolution function of $\Delta_j$. Then let
$$c_I = \sup_{K, N_0<\dots<N_K} (\sum_k |\sum_{N_{k-1}<j\le N_k} c_j|^r)^{1/r}$$
then it is not hard to see that
$$ 	\Big|\sup_{K, N_0<\dots<N_K} (\sum_k |\sum_{N_{k-1}<j\le N_k} \Delta_j f_j|^r)^{1/r} - c_I\Big|$$
$$\le \sup_{K, N_0<\dots<N_K} (\sum_k |\sum_{N_{k-1}<j\le N_k} (\Delta_j f_j - c_j)|^r)^{1/r} \,.$$
We then decompose 
$$\Delta_j f_j - c_j = g_j + b_j$$
where
$$(g_j, b_j) = \begin{cases}(0, \Delta_j f_j - c_j) & 2^j < \frac 1{|I|} \\
(\Delta_j (f_j 1_{3I}), \Delta_j (f_j 1_{(3I)^c})), & \text{otherwise}
\end{cases} \,.$$
It is not hard to see that  for any $y\in I$ we have
$$|b_j(y)| \le C \M_1 f_j(x) \min\Big[(2^{j}|I|)^{\epsilon}, (2^{j}|I|)^{-\epsilon}\Big] \,.$$
The parameter $\epsilon>0$ here depends on the decay of $\phi_j$ and its derivative. Now, by H\"older's inequality and the known Lebesgue case\footnote{Note that in the Lebesgue case, \eqref{e.lpfam} is equivalent to \eqref{e.rnot2} thanks to boundedness of the vector valued maximal function, this was observed in \cite{do-muscalu-thiele}.}  of \eqref{e.lpfam}, we have
$$\frac 1 {|I|} \int_I \sup_{K, N_0<\dots<N_K} (\sum_k |\sum_{N_{k-1}<j\le N_k}\Delta_j g_j|^r)^{1/r}$$
$$\le  \frac 1{|I|^{1/t}} \|\sup_{K, N_0<\dots<N_K} (\sum_k |\sum_{N_{k-1}<j\le N_k} \Delta_j g_j|^r)^{1/r}\|_t$$
$$\le C\frac 1{|I|^{1/t}} \|(\sum_j |g_j|^s)^{1/s}\|_t \le \frac 1{|I|^{1/t}} \|(\sum_j |f_j 1_{3I}|^s)^{1/s}\|_t \le \M_t(\f)(x) \,.$$
On the other hand,
$$\frac 1 {|I|}\int_I \sup_{K, N_0<\dots<N_K} (\sum_k |\sum_{N_{k-1}<j\le N_k} |b_j|^r)^{1/r}
\le \frac 1 {|I|} \int_I \sum_j |b_j(y)|dy$$
$$\le C \sum_j \min\Big[(2^j |I|)^{\epsilon}, (2^j |I|)^{-\epsilon}\Big] \M_1 f_j(x)$$
$$\le C\sup_j \M_1 f_j(x) \le C \M_1(\f)(x) \le C \M_t (\f)(x) \,.$$
\end{proof}

\section{The main argument and proof of Proposition~\ref{p.restrictedweaktype}}\label{s.mainargument}
Without loss of generality assume that $w(F)>0$ and $w(G)>0$ and
$$\max(w(F), w(G)) = 1 \,.$$
Recall that our aim is to find major subsets  of $F$ and $G$ respectively such that at least one of them has full measure, and if $|f|$ and $|g|$ are supported inside these sets and bounded above by $1$ then
\begin{equation}\label{e.BP}
B_{\P}(f, g) \le C w(F)^{1/p} w(G)^{1-1/p} \, \, 
\end{equation}
for all $p \in (q,\infty)$  such that $1/r>1/q-1/p$. The major subsets will be chosen using the weighted maximal function, see its definition in Section~\ref{s.notation}.\\

\noindent \underline{Case 1: $w(F)\le w(G)$.}
 
We choose $\widetilde F = F$ and  $\widetilde G = G\setminus \Omega$ with
$$\Omega :=  \{\M_{1,w}1_F > Cw(F)\}$$ 
and $C<\infty$ is sufficiently large such that $w(\Omega)<1/2$. 

Fix $q_0 \in (q,\infty)$ very close to $q$. We use the following estimate whose (rather standard) proof is included later: 
\begin{lemma}\label{l.eta} For any $\eta \in (\frac {2q_0}r, 1)$ there is a positive constant $\epsilon=\epsilon(\eta,q_0,r)>0$ such that
\begin{equation}\label{e.eta}
B_{\P}(f,g) \le C  \size(\P)^{1-\eta}  \density(\P)^{\epsilon} w(F)^{\eta/(2q_0)}  \,.
\end{equation}
Furthermore, if the elements of $\P$ are disjoint in the phase plane then a stronger variant of \eqref{e.eta} holds where $\widetilde{\density}(\P)$ is used in place of $\density(\P)$.
\end{lemma}
Below we show how Lemma~\ref{l.eta} implies the desired estimate \eqref{e.BP} using an argument from \cites{MTTBiCarl, MTTBiestFourier}. We decompose the original $\P= \bigcup_{k\ge 0} \P^{[k]}$ where
$$\P^{[k]} = \{P \in \P: 2^k \le 1+ \frac{\dist(I_P, \Omega^c)}{|I_P|} < 2^{k+1}\} \,.$$
Observe that if $P\in \P^{[k]}$ then $2^{k+2}I_P \cap \Omega^c \ne\emptyset$. Therefore, using Lemma~\ref{l.sizebound} we obtain
\begin{equation}\label{e.growsizebound}
\size(\P^{[k]}) \le C 2^{O(k)} w(F)^{1/q} \,.
\end{equation}
On the other hand, it is not hard to see that
$$\widetilde{\density}(\P^{[k]}) \le C 2^{-Dk/2} \,.$$
Now, observe that if $k\ge 1$  then the collection $\P^{[k]}$ can be decomposed into $O(1)$ bitile subcollections, such that for any two $P \ne P'$ in a subcollection we have $I_P\times  \omega_P \cap I_{P'} \times \omega_{P'} = \emptyset$. To see this, note that for $k\ge 1$ the length of any nested sequence in $\{I_P: P \in \P^{[k]} \}$ must be $O(1)$. It then follows that we can decompose $\P^{[k]}$ into $O(1)$ subcollections, in each collection the spatial intervals $I_P$ of two bitiles are either the same or disjoint, and via another decomposition (to ensure that any two different bitiles sharing the same spatial interval are far from each other in frequency) we can obtain $O(1)$ subcollections with the desired properties. 

Thus, for the purpose of proving \eqref{e.BP} we may assume without loss of generality  that for $k>k_0$ the elements of $\P^{[k]}$ are disjoint in the phase plane. For those $k$ we have 
$$B_{\P^{[k]}}(f,g)  \le C\size(\P^{[k]})^{1-\eta} [\widetilde{\density}(\P^{[k]})]^\epsilon w(F)^{\eta/(2q_0)}$$
$$\le C 2^{-D\epsilon k/2 } \size(\P^{[k]})^{1-\eta} w(F)^{\eta/(2q_0)} \qquad \text{(since $\supp(g)\subset \Omega^c$)}$$
$$\le  C  2^{-D\epsilon k/2} \Big[2^{O(k)} w(F)^{1/q}\Big]^{1-\eta} w(F)^{\eta/(2q_0)} \,.$$
Choosing $D$ large in the definition of density (certainly $D$ depends on $q,q_0,r,w$) we obtain
$$B_{\P^{[k]}}(f,g) \le C 2^{-\epsilon k} w(F)^{(1-\eta)/q + \eta/(2q_0)} \,, \ \ k> k_0 \,.$$
On the other hand for $0\le k<k_0$ disjointness may not be available, and we only have $\density(\P^{[k]}) = O(1)$, but since $k_0=O(1)$ we also have $\size(\P^{[k]})=O(w(F)^{1/q})$ from \eqref{e.growsizebound}. Using a similar argument as before, we obtain
$$B_{\P^{[k]}}(f,g) \le C w(F)^{(1-\eta)/q + \eta/(2q_0)} \,, \ \ k\le k_0 \,.$$
Thus, summing the above estimates over $k\ge 0$ we obtain
$$B_{\P}(f,g) \le C w(F)^{(1-\eta)/q + \eta/(2q_0)}  \,.$$
For any $p$ such that
$$\frac 1 p < \frac 1 q - \frac 1 r$$ 
we can choose $q_0$ sufficiently close to $q$ and $\eta$ sufficiently close to $2q_0/r$ (keeping $1>\eta > 2q_0/r$ and $q_0>q$) such that
$$(1-\eta)/q + \eta/(2q_0) > 1  / p \,.$$
The desired estimate \eqref{e.BP} now follows immediately, using $w(F)\le 1$.
$$B_{\P}(f,g) \le C w(F)^{1/p} \,.$$

\noindent \underline{Proof of Lemma~\ref{l.eta}:}

We show only the general case when $\P$ is arbitrary. An analogous argumet is used in the  case when any two elements of $\P$ are disjoint 
are disjoint in the phase plane, and the estimate is in terms of the improved density. The main difference is the use of the improved tree estimate (Lemma~\ref{l.improved-tree-est}) in place of the standard tree estimate (Lemma~\ref{l.tree-est}).

For convenience, we denote $S_1=\size(\P)$, $E_1=w(F)^{1/(2q_0)}$ and $D_1=\density(\P)$. Using Lemma~\ref{l.size} and Lemma~\ref{l.density} we can decompose  $\P = \bigcup_{n\in \Z}\P_n$ where each $\P_n$ is union of trees inside a tree collection $\T_n$, such that  
$$\sum_{T\in \T_n} w(I_T) \le C 2^{n}\,,$$
$$\size(\P_n) \le C 2^{-n/(2q_0)} E_1\,,  \ \ \density(\P_n) \le 2^{-n/r'} \,.$$
It then follows from the tree estimate \eqref{e.tree-est} (applied with $L^1$ norm) that
$$B_{\P}(f,g) \le C\sum_{n\in \Z} \sum_{T\in \T_n}w(I_T)\size(T)\density(T)$$
$$\le C\sum_{n\in \Z} 2^n \min(S_1, 2^{-\frac n{2q_0}}E_1) \min(D_1, 2^{-\frac n{r'}}) $$

It follows that for $\alpha, \beta\in [0,1]$ we have
$$B_{\P}(f,g) \le C S_1 D_1\sum_{n\in Z} \min(1, 2^{-\frac{n}{2q_0}} E_1 S_1^{-1})^{\alpha} \min(1, 2^{-n/r'}D_1^{-1})^{\beta}$$
$$\le CS_1 D_1 \sum_{n\in \Z} 2^n \min\Big(1, 2^{-n K} (E_1/S_1)^{\alpha}D_1^{-\beta}\Big)\,,$$
$$K:= \frac{\alpha}{2q_0} + \frac{\beta}{r'} \,.$$
Under the assumption $r>2q$ we can choose $q_0>q$ such that $r>2q_0$. Then we can find $\alpha,\beta \in [0,1]$ such that 
\begin{equation}\label{e.alphabetacondition}
\frac{\alpha}{2q_0} + \frac{\beta}{r'} > 1 \,.
\end{equation}
We then obtain a two-sided geometric series which is bounded above by its largest term. Thus
$$B_{\P}(f,g) \le C S_1 D_1 (E_1/S_1)^{\alpha/K} D_1^{-\beta/K}  = C S_1^{1-\alpha/K} E_1^{\alpha/K} D_1^{1-\beta/K} \,.$$
Let $\eta=\alpha/K$, we have $\eta \in (\frac{2q_0}r,1)$ and in fact varying $\alpha,\beta \in [0,1]$ respecting the condition \eqref{e.alphabetacondition} we can obtain any value of $\eta$ in $(2q_0/r, 1)$. Furthermore
$$\epsilon:=1-\frac{\beta}{K} = 1-r'(1-\frac{\eta}{2q_0}) = \frac{r'}{2q_0}(\eta-\frac{2q_0}r) > 0\,,$$
giving the desired estimate \eqref{e.eta}. This completes the proof of Lemma~\ref{l.eta}.\\

\noindent \underline{Case 2: $w(F)> w(G)$.} We will choose $\widetilde G = G$ and $\widetilde F=F\setminus \Omega$ where
$$\Omega = \{\M_{1,w}1_G > C w(G)\}$$
where $C<\infty$ is sufficiently large such that $w(\Omega)<1/2$.
We will use the following estimate, whose proof is included later:
\begin{lemma}\label{l.delta} Suppose that $\density(\P) \le M w(G)^{1/r'}$ for some $M\ge 1$. Then for any $p<\infty$   there exists a constant $\delta =\delta(p,q,w,r)>0$ such that
\begin{equation}\label{e.delta}
\B_{\P}(f,g) \le CM \size(\P)^\delta  w(G)^{1/p' - 1/r'} \,.
\end{equation}
\end{lemma}
Below we show how Lemma~\ref{l.delta} implies the desired estimate \eqref{e.BP}. Decompose $\P$ into $\bigcup_{h\ge 0}\P^{[h]} $ where
$$\P^{[h]} = \{P \in \P: 2^h \le 1+ \frac{\dist(I_P, \Omega^c)}{|I_P|} < 2^{h+1} \}$$ 
We verify below that
\begin{equation}\label{e.densityPh}
\density(\P^{[h]}) \le C 2^{O(h)} \Big[\sup_{x\in \Omega^c} (\M_{1,w}1_G)(x)\Big]^{1/r'} \le C 2^{O(h)} w(G)^{1/r'}\,,
\end{equation}
here the implicit constant in $O(h)$ depends on the doubling exponent $\gamma$ of $w$. Indeed, let $T$ be any non-empty tree in $\P^{[h]}$. Then it is clear that
$$1+\frac{\dist(I_T,\Omega^c)}{|I_T|} \le 2^{h+1} \,.$$
We then enlarge $I_T$ by a factor of $O(2^h)$ to obtain an interval $J$ such that $J\cap \Omega^c\ne\emptyset$, clearly $w(J)\le C 2^{O(h)} w(I_T)$ and therefore
$$(\frac{1}{w(I_T)} \int \widetilde \chi_{I_T}^D |g|^{r'} w)^{1/r'} \le C2^{O(h)} \Big[\inf_{x\in J} \M_{1,w} 1_G(x) \Big]^{1/r'}$$
from which the estimate \eqref{e.densityPh} follows immediately.

On the other hand, using $\supp(f) \subset \Omega^c$, it follows from Lemma~\ref{l.sizebound} that
$$\size(\P^{[h]}) \le C_N 2^{-Nh} $$
 for any $N>0$. Take $N$ very large in the above estimate, it follows from \eqref{e.delta} and \eqref{e.densityPh} that
$$\B_{\P^{[h]}}(f,g) \le C 2^{-h} w(G)^{1/r'} w(G)^{1/p' - 1/r'}  = C 2^{-h} w(G)^{1/p'}\,,$$
and \eqref{e.BP} now follows from summing these estimates over $h\ge 0$.\\

\noindent \underline{Proof of Lemma~\ref{l.delta}:} 
Fix $q_0 \in (q,\infty)$. Using Lemma~\ref{l.size} and Lemma~\ref{l.density}, we can decompose $\P=\bigcup_{n\in \Z} \P_n$ where $\P_n$ is the union of trees from a tree collection $\T_n$, such that
$$\sum_{T\in \T_n} w(I_T) \le 2^n\,,$$
$$\size(\P_n) \le C 2^{-\frac n {2q_0}}\,, \ \ \density(\P_n) \le C 2^{-n/r'}w(G)^{1/r'} \,.$$
We use Lemma~\ref{l.size} again and decompose $\P_n = \bigcup_{m\ge 0} \P_{n,m}$ where $\P_{n,m}$ is the union of trees from a tree collection $\T_{n,m}$ such that
$$\size(\P_{n,m}) \le C 2^{-(n+m)/(2q_0)}\,,$$
$$\sum_{T\in \T_{n,m}} w(I_T) \le C \sum_{T\in \T_n} w(I_T) \le C 2^n \ \  ,$$ 
$$\|\sum_{T\in \T_{n,m}} 1_{2^k I_T} \|_{L^p(w)} \le C 2^{O(k)} 2^{n+m} w(F)^{1/p} = C 2^{O(k)}2^{n+m} \,. $$
In particular, it follows from the doubling property of $w$ that
$$\|\sum_{T\in \T_{n,m}} 1_{2^k I_T} \|_{L^1(w)} \le C 2^{ \gamma k} 2^n \,.$$
By interpolation, it follows that for any $1<p<\infty$ and any $\epsilon>0$ we have
\begin{equation}\label{e.Lpcounting}
\|\sum_{T\in \T_{n,m}} 1_{2^k I_T} \|_{L^{p-\epsilon}(w)} \le C 2^{O(k)} 2^{m/p'} 2^n \,.
\end{equation}
Here, the implicit constant in $O(k)$ may depend on $p,\epsilon,w$. For convenience, for any $k\ge 0$ let $N^{[k]}_{n,m}$ denote the counting function
$$N^{[k]}_{n,m} = \sum_{T\in \T_{n,m}} 1_{2^k I_T} \,.$$
Decomposing $1 = 1_{I_T}  + \sum_{k\ge 0} (1_{2^{k+1}I_T}  - 1_{2^k I_T})$ for each $T$ and applying H\"older's inequality, we obtain
$$B_{\P_{n,m}}(f,g) \le C\sum_{k\ge -1} B_k(n,m)$$
$$B_{-1}(n,m):=\int (N^{[0]}_{n,m})^{1/r} (\sum_{T\in \T_{n,m}}|1_{I_T} gC_Tf|^{r'})^{1/r'} w dx\,,$$
$$B_k(n,m):= \int (N^{[k+1]}_{n,m} - N^{[k]}_{n,m})^{1/r} (\sum_{T\in \T_{n,m}}|1_{2^{k+1}I_T\setminus 2^k I_T} gC_Tf|^{r'})^{1/r'} w dx\,, \ \ k\ge 0 \,.$$

\noindent \underline{Estimate for $\sum_{n,m} B_{-1}(n,m)$:} Fix $p<\infty$ very large and $\epsilon>0$ very small, such that in particular $p-\epsilon>r$. Apply H\"older's inequality we obtain
$$B_{-1}(n,m)  \le C  \|(N^{[0]}_{n,m})^{1/r}\|_{L^{p-\epsilon}(w)} \|(\sum_{T\in \T_{n,m}}|1_{I_T} gC_Tf|^{r'})^{1/r'} \|_{L^{(p-\epsilon)'}(w)}  \,.$$
Using \eqref{e.Lpcounting}, the first factor can be rewritten and estimated by
$$\|\sum_{T\in \T_{n,m}}1_{I_T}\|_{L^{(p-\epsilon)/r}(w)}^{1/r} \le C 2^{n/r} 2^{(\frac 1 r - \frac 1 p)m}$$
using $p-\epsilon > r$. The second factor is supported inside $supp(g) \subset G$, thus it can be bounded above by
$$\le C w(G)^{\frac 1{(p-\epsilon)'} - \frac 1 {r'}}\|(\sum_{T\in \T_{n,m}}|1_{I_T} gC_Tf|^{r'})^{1/r'} \|_{L^{r'}(w)} $$
$$= C w(G)^{\frac 1{(p-\epsilon)'} - \frac 1 {r'}} (\sum_{T\in \T_{n,m}}\|1_{I_T} gC_Tf \|_{L^{r'}(w)}^{r'} )^{1/r'} $$
Using the tree estimate \eqref{e.tree-est-core}, we can bound the above expression by
$$\le C w(G)^{\frac 1{(p-\epsilon)'} - \frac 1 {r'}}  \Big[\sum_{T\in \T_{n,m}}w(I_T)\Big]^{1/r'}\size(\P_{n,m}) \density(\P_{n,m})$$
$$\le C w(G)^{\frac 1{(p-\epsilon)'} - \frac 1 {r'}} \Big[2^{\frac n{r'}}\Big] \Big[2^{-(n+m)(\frac{1}{2q_0} - \delta)}  \size(\P)^{\delta} \Big]\min\Big(2^{-n/r'} w(G)^{1/r'}, M w(G)^{1/r'} \Big) \,.$$
here $\delta \in (0,\frac 1{2q_0})$ very small to be chosen later. Since $M\ge 1$, it follows that
$$\sum_{m\ge 0} B_{-1}(m,n)$$
$$\le CM w(G)^{\frac 1{(p-\epsilon)'}} \size(\P)^{\delta} \sum_{m\ge 0}  \Big[2^{ \frac nr} 2^{(\frac 1 r - \frac 1 p)m} \Big]\Big[ 2^{\frac n{r'}} 2^{-(n+m)(\frac{1}{2q_0} - \delta)} \min(2^{-n/r'}, 1)  \Big]\,.$$
Since $r>2q$ we can always choose $q_0>q$ such that $r>2q_0$, and then choose $\delta>0$ depends on $q_0,r$ such that 
$$\frac 1 r -  (\frac 1 {2q_0} -\delta)< 0\,,$$
which implies $1/r - 1/p -  1/(2q_0) +\delta  < 0$. Therefore the above summation over $m\ge 0$ converges, and
$$\sum_n \sum_{m\ge 0} B_{-1}(m,n) \le   CM w(G)^{\frac 1{(p-\epsilon)'}}  \size(\P)^{\delta} \sum_{n\in\Z} 2^{n(\frac 1 r - \frac 1 {2q_0} + \delta)} \min(1, 2^{\frac n {r'}}) \,.$$
Since 
$$\frac 1 r - \frac 1{2q_0} < 0 < \frac 1 r - \frac 1 {2q_0} + \frac 1 {r'}$$
we can refine our previous choice of $\delta = \delta(q_0,r)>0$ such that the above estimate of $\sum_n \sum_{m\ge 0} B_{-1}(m,n)$ remains a two-sided geometric series. It follows that
$$\sum_n \sum_{m\ge 0} B_{-1}(m,n)  \le CM w(G)^{\frac 1{(p-\epsilon)'}} \size(\P)^{\delta} \,.$$
Since we can choose $p<\infty$ arbitrarily large and since $w(G)\le 1$, it follows that
$$\sum_n \sum_{m\ge 0} B_{-1}(m,n) \le CM w(G)^{1/p'} \size(\P)^{\delta}$$
for any $ p <\infty$.\\

\noindent \underline{Estimate for $\sum_{n,m} B_k(n,m)$:} The argument is similar to the above estimate for the sum of $B_{-1}(n,m)$, with the following difference: we will collect some power $2^{k}$, and we will gain the decay factor $2^{-Nk}$ from the tree estimate \eqref{e.tree-est-tail} where $N$ could be chosen arbitrarily large. We obtain, via a similar argument and by choosing $N$ large enough, the following estimate
$$\sum_n \sum_{m\ge 0} B_{k}(m,n) \le C 2^{-k} Mw(G)^{1/p'}\size(\P)^{\delta} $$
for any $p <\infty$.  

Summing over $k\ge -1$, we obtain the desired estimate \eqref{e.delta}. This completes the proof of Lemma~\ref{l.delta}.

\end{document}